\documentclass[12pt]{article}

\usepackage{amsfonts,amsmath,amsthm,amssymb,amscd}
\usepackage[dvips]{graphicx}

\theoremstyle{definition}
\newtheorem{theorem}{Theorem}
\newtheorem{lemma}[theorem]{Lemma}
\newtheorem{proposition}[theorem]{Proposition}
\newtheorem{corollary}[theorem]{Corollary}

\numberwithin{equation}{section}
\numberwithin{theorem}{section}

\begin{document}

\begin{center}
{\bf{\Large Analogues of Jacobi's derivative formula III}}
\end{center}

\begin{center}
By Kazuhide Matsuda
\end{center}

\begin{center}
Faculty of Fundamental Science, National Institute of Technology, Niihama College,\\
7-1 Yagumo-chou, Niihama, Ehime, Japan 792-8580. \\
E-mail: matsuda@sci.niihama-nct.ac.jp  \\
Fax: 81-0897-37-7809
\end{center}

\noindent
{\bf Abstract}
In this paper,
we realize high-level versions of Jacobi's derivative formula to all the rational characteristics corresponding to level 
$k \,\,(k=3,4,5,6).$ 
For this purpose, 
we propose the method to obtain derivative formulas by means of the residue theorem. 
We believe that this method can be also applied to all the rational characteristics corresponding to level $k\ge 7.$ 

\section{Introduction}
Throughout this paper,
for each positive integer $k$
set $\zeta_k=\exp(2\pi i /k).$ 
For the positive integers $j,k,$ and $n,$
$d_{j,k}(n)$ denotes the number of positive divisors $d$ of $n$ such that $d\equiv j \,\,(\mathrm{mod} \,k),$
and
$d_{j,k}^{*}(n)$ denotes the number of positive divisors $d$ of $n$ such that $d\equiv j \,\,(\mathrm{mod} \,k)$ and $n/d$ is odd.
Moreover, let the upper half plane be defined by
$$
\mathbb{H}^2=\{ \tau \in\mathbb{C} \, | \, \Im \tau >0 \}.
$$
\par
Following Farkas and Kra \cite{Farkas-Kra},
we introduce the {\it theta function with characteristics,}
which is defined by
\begin{equation*}
\theta
\left[
\begin{array}{c}
\epsilon \\
\epsilon^{\prime}
\end{array}
\right] (\zeta, \tau) :=\sum_{n\in\mathbb{Z}} \exp
\left(2\pi i\left[ \frac12\left(n+\frac{\epsilon}{2}\right)^2 \tau+\left(n+\frac{\epsilon}{2}\right)\left(\zeta+\frac{\epsilon^{\prime}}{2}\right) \right] \right), 
\end{equation*}
where $\epsilon, \epsilon^{\prime}\in\mathbb{R}, \, \zeta\in\mathbb{C},$ and $\tau\in\mathbb{H}^{2}.$
The {\it theta constants} are given by
\begin{equation*}
\theta
\left[
\begin{array}{c}
\epsilon \\
\epsilon^{\prime}
\end{array}
\right]
:=
\theta
\left[
\begin{array}{c}
\epsilon \\
\epsilon^{\prime}
\end{array}
\right] (0, \tau).
\end{equation*}
Let us denote the {\it theta derivatives} by
\begin{equation*}
\theta^{\prime}
\left[
\begin{array}{c}
\epsilon \\
\epsilon^{\prime}
\end{array}
\right]
:=\left.
\frac{\partial}{\partial \zeta}
\theta
\left[
\begin{array}{c}
\epsilon \\
\epsilon^{\prime}
\end{array}
\right] (\zeta, \tau)
\right|_{\zeta=0}.
\end{equation*}
{\it Jacobi's derivative formula} is then given by
\begin{equation}
\label{eqn:Jacobi-derivative}
\theta^{\prime}
\left[
\begin{array}{c}
1 \\
1
\end{array}
\right]
=
-\pi
\theta
\left[
\begin{array}{c}
0 \\
0
\end{array}
\right]
\theta
\left[
\begin{array}{c}
1 \\
0
\end{array}
\right]
\theta
\left[
\begin{array}{c}
0 \\
1
\end{array}
\right].
\end{equation}
We note that Farkas and Kra's notation is different from Mumford's notation, 
\begin{equation}
\label{eqn:Mumford-theta}
\theta_{a,b}(z,\tau)
:=\sum_{n\in\mathbb{Z}} \exp
\left(\pi i\left(n+a\right)^2 \tau+2 \pi i\left(n+a\right)\left(z+b\right)  \right) \,\, \mathrm{for} \,\,z\in\mathbb{C} \,\,\mathrm{and} \,\,\tau\in\mathbb{H}^2.
\end{equation}
Farkas and Kra \cite{Farkas-Kra}
developed the theory of theta function with {\it rational characteristics},
i.e.,
the case where $\epsilon$ and $\epsilon^{\prime}$ are both rational numbers,
and 
applied the theory to number theory such as partition number. 
\par
In this paper,
we treat the following problem of Mumford \cite[pp. 117]{Mumford}:
 \begin{quotation}
Can Jacobi's formula be generalized, e.g., to
$$
\left(
\frac{\partial}{\partial z} \theta_{a,b}
\right)(0,\tau)
=
\left\{
\textrm{cubic polynomial in}
\,\,
\theta_{c,d}
\,\,
\textrm{'s}
\right\}
$$
for all $a,b \in\mathbb{Q}$?
Similarly, are there generalizations of Jacobi's formula with higher-order differential operators?
\end{quotation}
The aim of this paper is to deal with Mumford's problem for all the rational characteristics corresponding to level 
$k \,\,(k=3,4,5,6).$ While our derivative formulas are not cubic polynomials but rational expressions in theta constants,  
they yield many product-series identities.  
\par
Our motivation lies in ordinary differential equations (ODEs) satisfied by modular forms, 
whose typical example is Ramanujan's coupled system of ODEs, 
\begin{equation}
\label{eqn:Ramanujan-ODE}
q\frac{E_2}{dq}=\frac{(E_2)^2-E_4  }{12  }, \,\,
q\frac{E_4}{dq}=\frac{E_2 E_4-E_6  }{3  }, \,\,
q\frac{E_6}{dq}=\frac{E_2 E_6-(E_4)^2  }{2  }, 
\end{equation}
where 
\begin{align*}
E_2(q)=E_2(\tau)&:=1-24\sum_{n=1}^{\infty} \frac{nq^n}{1-q^n}, \,\,
E_4(q)=E_4(\tau):=1+240\sum_{n=1}^{\infty} \frac{n^3 q^n}{1-q^n}, \\
E_6(q)=E_4(\tau)&:=1-504\sum_{n=1}^{\infty} \frac{n^5 q^n}{1-q^n}, \,\,q=\exp(2 \pi i \tau).  
\end{align*}
In \cite{Matsuda4}, 
by means of Farkas and Kra's theory, 
the author derived coupled systems of ODEs satisfied by the cubic theta functions, 
\begin{align*}
a(q)=&\sum_{m,n\in\mathbb{Z}} q^{m^2+mn+n^2}, \,\,
b(q)=\sum_{m,n\in\mathbb{Z}} \omega^{n-m} q^{m^2+mn+n^2}, \\
c(q)=&\sum_{m,n\in\mathbb{Z}} q^{(n+\frac13)^2+(n+\frac13)(m+\frac13)+(m+\frac13)^2}, \,\,\omega=e^{\frac{2\pi i}{3}}, \,\,|q|<1.
\end{align*}
In \cite{Matsuda3}, 
applying the drivative formulas, he showed that 
some ratios of theta constants satisfy Riccati equations. 
The study of derivative formulas is expected to lead to ODEs satisfied by modular forms. 
\par
In \cite{Matsuda1}, 
the author expressed
$
\theta^{\prime}
\left[
\begin{array}{c}
1 \\
\frac12
\end{array}
\right],
\theta^{\prime}
\left[
\begin{array}{c}
1 \\
\frac14
\end{array}
\right],
$
\normalsize{
and
}
$
\theta^{\prime}
\left[
\begin{array}{c}
1 \\
\frac34
\end{array}
\right]
$
by the theta constants with rational characteristics.
For this purpose,
he used the following arithmetical formulas:
\begin{theorem}
(Sums of squares)
\label{thm:sums-of-squares}
{\it
For each positive integer $n\in\mathbb{N},$ the following holds: 
\begin{equation}
\label{eqn:2-squares}
S_2(n)=\sharp \{(x,y)\in\mathbb{Z}^2  \,| \, n=x^2+y^2 \}
        =4(d_{1,4}(n)-d_{3,4}(n))
\end{equation}
and
\begin{equation}
\label{eqn:1,2-squares}
S_{1,2}(n)=\sharp \{(x,y)\in\mathbb{Z}^2  \,| \, n=x^2+2y^2 \}
            =2(d_{1,8}(n)+d_{3,8}(n)-d_{5,8}(n)-d_{7,8}(n)).
\end{equation}
}
\end{theorem}
\noindent
We note that Zemel \cite{Zemel} used binary quadratic forms, $x^2+y^2, x^2+2y^2$ and $x^2+xy+y^2$ 
to treat all the rational characteristics corresponding to level $k \,\, (k=3,4,6,8)$ and 
obtain derivative formulas, which are cubic polynomials in the theta constants with rational characteristics. 
For the proof of Theorem \ref{thm:sums-of-squares}, see \cite[pp. 56,74]{Berndt} or \cite[pp. 68]{Dickson}.
\par
In \cite{Matsuda2}, the author treated
the case where $\epsilon=1$ and $\epsilon^{\prime}\in\mathbb{Q}$, and obtained derivative formulas. 
For this purpose,
we mainly used Liouville's theorem:
\begin{theorem}
\label{thm-Liouville-elliptic-function}
(Liouville's theorem)
{\it
There exists no nonconstant elliptic function without poles, i.e.,
a holomorphic elliptic function is a constant.
}
\end{theorem}
\noindent
\par
In this paper,
we obtain high-level versions of Jacobi's derivative formula to all the rational characteristics corresponding to level $k \,\,(k=3,4,5,6).$ 
Moreover, for level 8, we investigate derivative formulas from the viewpoint of the number of expressions of a positive integer $n$ as sums of squares, or sums of triangular numbers, 
where the {\it squares} are $x^2 \,\,(x\in\mathbb{Z})$ and the {\it triangular numbers} are $t_x=x(x+1)/2$ for $x\in\mathbb{Z}.$ 
For this purpose, we mainly use the residue theorem:
\begin{theorem}
\label{thm-fundamental-elliptic-function}
(The residue theorem)
{\it
The sum of all the residues of an elliptic function in the fundamental parallelogram is zero.
}
\end{theorem}
\par
The remainder of this paper is organized as follows. Section \ref{sec:properties} reviews the properties of theta functions.
Sections \ref{sec:derivative-level4}, \ref{sec:derivative-level5}, \ref{sec:derivative-level6}, and \ref{sec:derivative-level3} treat derivative formulas to to all the rational characteristics corresponding to level 
$k=4,5,6,3.$ 
Section \ref{sec:derivative-level8} deals with derivative formulas to some rational characteristics corresponding  to level $k=8.$
\par
In particular, in Section \ref{sec:derivative-level6}
we derive the K\"ohler-Macdonald identity:
\begin{equation*}
\frac{\eta^5(2\tau)}{\eta^2(\tau)}
=
\sum_{n=1}^{\infty} (-1)^{n-1}n\left(\frac{n}{3}\right) \exp \left(\frac{2\pi i n^2\tau}{3} \right),
\end{equation*}
where 
the {\it Dedekind eta function} is defined by
\begin{equation*}
\eta(\tau)=q^{\frac{1}{24}} \prod_{n=1}^{\infty} (1-q^n), \,\,q=\exp(2\pi i \tau) \,\,\text{for} \,\,\tau\in\mathbb{H}^2,
\end{equation*}
and 
\begin{equation*}
\left(\frac{n}{3}\right)
=
\begin{cases}
+1, \,\,&\text{if}  \,\,n\equiv 1  \,\,(\mathrm{mod} \, 3),  \\
-1, \,\,&\text{if}  \,\,n\equiv -1 \,\,(\mathrm{mod} \, 3),  \\
0, \,\,&\text{if}  \,\,n\equiv 0 \,\,(\mathrm{mod} \, 3).  \\
\end{cases}
\end{equation*}
See the bibliographical notes of Farkas and Kra \cite[pp. 518-519]{Farkas-Kra}.
Further, in Section \ref{sec:derivative-level3}, we show a relation between derivative formulas and theta constant identities,
and in Section \ref{sec:derivative-level8},
we use relations between
sums of squares and sums of triangular numbers. 
It is noted that
Adiga et al. \cite{Adiga-Cooper-Han}, Barrucand et al. \cite{Barrucand-Cooper-Hirschhorn}, and Baruah et al. \cite{Baruah-Cooper-Hirschhorn}
derived explicit relations between
sums of squares and sums of triangular numbers.
\par
In this series of papers, 
we present three different methods for obtaining derivative formulas:
(1) formulas on expressions in sums of squares, (2) Liouville's theorem, and (3) the residue theorem. 
While the methods (1) and (2) cannnot be applied to all the rational characteristics, we believe that the method (3) can be applied to all the rational characteristics.

\section{Properties of the theta functions}
\label{sec:properties}

\subsection{Basic properties}
We first note that
for $m,n\in\mathbb{Z},$
\begin{equation}
\label{eqn:integer-char}
\theta
\left[
\begin{array}{c}
\epsilon \\
\epsilon^{\prime}
\end{array}
\right] (\zeta+n+m\tau, \tau) =
\exp \left\{2\pi i \left[\frac{n\epsilon-m\epsilon^{\prime}}{2}-m\zeta-\frac{m^2\tau}{2}\right] \right\}
\theta
\left[
\begin{array}{c}
\epsilon \\
\epsilon^{\prime}
\end{array}
\right] (\zeta,\tau)
\end{equation}
and
\begin{equation}
\theta
\left[
\begin{array}{c}
\epsilon +2m\\
\epsilon^{\prime}+2n
\end{array}
\right]
(\zeta,\tau)
=\exp(\pi i \epsilon n)
\theta
\left[
\begin{array}{c}
\epsilon \\
\epsilon^{\prime}
\end{array}
\right]
(\zeta,\tau).
\end{equation}
Furthermore,
it is easy to see that
\begin{equation*}
\theta
\left[
\begin{array}{c}
-\epsilon \\
-\epsilon^{\prime}
\end{array}
\right] (\zeta,\tau)
=
\theta
\left[
\begin{array}{c}
\epsilon \\
\epsilon^{\prime}
\end{array}
\right] (-\zeta,\tau)
\,\,
\mathrm{and}
\,\,
\theta^{\prime}
\left[
\begin{array}{c}
-\epsilon \\
-\epsilon^{\prime}
\end{array}
\right] (\zeta,\tau)
=
-
\theta^{\prime}
\left[
\begin{array}{c}
\epsilon \\
\epsilon^{\prime}
\end{array}
\right] (-\zeta,\tau).
\end{equation*}
\par
For $m,n\in\mathbb{R},$
we see that
\begin{align}
\label{eqn:real-char}
&\theta
\left[
\begin{array}{c}
\epsilon \\
\epsilon^{\prime}
\end{array}
\right] \left(\zeta+\frac{n+m\tau}{2}, \tau\right)   \notag\\
&=
\exp(2\pi i)\left[
-\frac{m\zeta}{2}-\frac{m^2\tau}{8}-\frac{m(\epsilon^{\prime}+n)}{4}
\right]
\theta
\left[
\begin{array}{c}
\epsilon+m \\
\epsilon^{\prime}+n
\end{array}
\right]
(\zeta,\tau).
\end{align}
We note that
$\theta
\left[
\begin{array}{c}
\epsilon \\
\epsilon^{\prime}
\end{array}
\right] \left(\zeta, \tau\right)$ has the only zero with order one in the fundamental parallelogram,
which is given by
$$
\zeta=\frac{1-\epsilon}{2}\tau+\frac{1-\epsilon^{\prime}}{2}.
$$

\subsection{Jacobi's triple product identity}

All the theta functions have infinite product expansions, which are given by
\begin{align}
\theta
\left[
\begin{array}{c}
\epsilon \\
\epsilon^{\prime}
\end{array}
\right] (\zeta, \tau) &=\exp\left(\frac{\pi i \epsilon \epsilon^{\prime}}{2}\right) x^{\frac{\epsilon^2}{4}} z^{\frac{\epsilon}{2}}    \notag  \\
                           &\quad
                           \displaystyle \prod_{n=1}^{\infty}(1-x^{2n})(1+e^{\pi i \epsilon^{\prime}} x^{2n-1+\epsilon} z)(1+e^{-\pi i \epsilon^{\prime}} x^{2n-1-\epsilon}/z),  \label{eqn:Jacobi-triple}
\end{align}
where $x=\exp(\pi i \tau)$ and $z=\exp(2\pi i \zeta).$

\subsection{Lemma of Farkas and Kra}

We recall the following lemma of Farkas and Kra \cite[pp. 78]{Farkas-Kra}.

\begin{lemma}
\label{lem:Farkas-Kra}
{\it
For
all characteristics
$
\left[
\begin{array}{c}
\epsilon \\
\epsilon^{\prime}
\end{array}
\right],
\left[
\begin{array}{c}
\delta \\
\delta^{\prime}
\end{array}
\right]
$
and
all $\tau\in\mathbb{H}^2,$
we have
\begin{align*}
&\theta
\left[
\begin{array}{c}
\epsilon \\
\epsilon^{\prime}
\end{array}
\right](0,\tau)
\theta
\left[
\begin{array}{c}
\delta \\
\delta^{\prime}
\end{array}
\right](0,\tau)  \\
=&
\theta
\left[
\begin{array}{c}
\frac{\epsilon+\delta}{2} \\
\epsilon^{\prime}+\delta^{\prime}
\end{array}
\right](0,2\tau)
\theta
\left[
\begin{array}{c}
\frac{\epsilon-\delta}{2} \\
\epsilon^{\prime}-\delta^{\prime}
\end{array}
\right](0,2\tau)
+
\theta
\left[
\begin{array}{c}
\frac{\epsilon+\delta}{2}+1 \\
\epsilon^{\prime}+\delta^{\prime}
\end{array}
\right](0,2\tau)
\theta
\left[
\begin{array}{c}
\frac{\epsilon-\delta}{2}+1 \\
\epsilon^{\prime}-\delta^{\prime}
\end{array}
\right](0,2\tau).
\end{align*}
}
\end{lemma}

\section{Derivative formulas of level 4}
\label{sec:derivative-level4}

From the discussion of Farkas and Kra \cite[pp. 185-187]{Farkas-Kra},
we have only to consider the case where
\begin{equation*}
(\epsilon, \epsilon^{\prime})=(1,1/2), (0,1/2), (1/2,0), (1/2,1), (1/2,1/2), (1/2,3/2).
\end{equation*}

\subsection{Derivative formulas for $(\epsilon, \epsilon^{\prime})=(1,1/2), (0,1/2)$}

\begin{theorem}
\label{thm:analogue-Jacobi-(1,1/2), (0,1/2)}
{\it
For every $\tau\in\mathbb{H}^2,$ we have
\begin{equation}
\label{eqn-derivative-(1,1/2)}
\theta^{\prime}
\left[
\begin{array}{c}
1 \\
\frac12
\end{array}
\right](0,\tau)
=
\frac14
\frac
{
\theta^{\prime}
\left[
\begin{array}{c}
1 \\
1
\end{array}
\right](0,\tau)
\theta^3
\left[
\begin{array}{c}
1 \\
0
\end{array}
\right](0,\tau)
}
{
\theta^3
\left[
\begin{array}{c}
1 \\
\frac12
\end{array}
\right](0,\tau)
}
=
-\pi
\theta^2
\left[
\begin{array}{c}
0 \\
0
\end{array}
\right](0,2\tau)
\theta
\left[
\begin{array}{c}
1 \\
\frac12
\end{array}
\right](0,\tau)
\end{equation}
and
\begin{equation}
\label{eqn-derivative-(0,1/2)}
\theta^{\prime}
\left[
\begin{array}{c}
0 \\
\frac12
\end{array}
\right](0,\tau)
=
\frac14
\frac
{
\theta^{\prime}
\left[
\begin{array}{c}
1 \\
1
\end{array}
\right](0,\tau)
\theta^3
\left[
\begin{array}{c}
1 \\
0
\end{array}
\right](0,\tau)
}
{
\theta^3
\left[
\begin{array}{c}
0 \\
\frac12
\end{array}
\right](0,\tau)
}
=
-\pi
\theta^2
\left[
\begin{array}{c}
1 \\
0
\end{array}
\right](0,2\tau)
\theta
\left[
\begin{array}{c}
0 \\
\frac12
\end{array}
\right](0,\tau).
\end{equation}
}
\end{theorem}

\begin{proof}
Consider the following elliptic functions:
\begin{equation*}
\varphi(z)
=
\frac
{
\theta^3
\left[
\begin{array}{c}
1 \\
1
\end{array}
\right](z, \tau)
}
{
\theta^2
\left[
\begin{array}{c}
1 \\
\frac12
\end{array}
\right](z, \tau)
\theta
\left[
\begin{array}{c}
1 \\
0
\end{array}
\right](z, \tau)
}
\,\,
\mathrm{and}
\,\,
\psi(z)=
\frac
{
\theta^3
\left[
\begin{array}{c}
1 \\
1
\end{array}
\right](z, \tau)
}
{
\theta^2
\left[
\begin{array}{c}
0 \\
\frac12
\end{array}
\right](z, \tau)
\theta
\left[
\begin{array}{c}
1 \\
0
\end{array}
\right](z, \tau)
}.
\end{equation*}
\par
We first note that
in the fundamental parallelogram,
the poles of $\varphi(z)$ are $z=1/4$ and $z=1/2.$
Direct calculation yields
\begin{equation*}
\mathrm{Res}\left(\varphi(z), \frac14\right)
=
4
\frac
{
\theta
\left[
\begin{array}{c}
1 \\
\frac12
\end{array}
\right]
\theta^{\prime}
\left[
\begin{array}{c}
1 \\
\frac12
\end{array}
\right]
}
{
\theta^{\prime}
\left[
\begin{array}{c}
1 \\
1
\end{array}
\right]^2
}
\,\,
\mathrm{and}
\,\,
\mathrm{Res}\left(\varphi(z), \frac12 \right)
=
-
\frac
{
\theta^3
\left[
\begin{array}{c}
1 \\
0
\end{array}
\right]
}
{
\theta^2
\left[
\begin{array}{c}
1 \\
\frac12
\end{array}
\right]
\theta^{\prime}
\left[
\begin{array}{c}
1 \\
1
\end{array}
\right]
}.
\end{equation*}
Since
$
\mathrm{Res}\left(\varphi(z), \frac14\right)+\mathrm{Res}\left(\varphi(z), \frac12 \right)=0,
$
Eq. (\ref{eqn-derivative-(1,1/2)}) can be obtained.
The second equality follows from Jacobi's triple product identity (\ref{eqn:Jacobi-triple}).
\par
Eq. (\ref{eqn-derivative-(0,1/2)}) can be obtained from $\psi(z)$ in the same way.
\end{proof}

\subsection{Derivative formulas for $(\epsilon, \epsilon^{\prime})=(1/2,1), (1/2, 0)$}

\begin{theorem}
\label{thm:analogue-Jacobi-1/2-1,0}
{\it
For every $\tau\in\mathbb{H}^2,$ we have
\begin{equation*}
\theta^{\prime}
\left[
\begin{array}{c}
\frac12 \\
1
\end{array}
\right]
=
\frac{\zeta_4}{4}
\frac
{
\theta^{\prime}
\left[
\begin{array}{c}
1 \\
1
\end{array}
\right]
\theta^3
\left[
\begin{array}{c}
0 \\
1
\end{array}
\right]
}
{
\theta^3
\left[
\begin{array}{c}
\frac12 \\
1
\end{array}
\right]
}
\,\,
\text{and}
\,\,
\theta^{\prime}
\left[
\begin{array}{c}
\frac12 \\
0
\end{array}
\right]
=
\frac{\zeta_4^3}{4}
\frac
{
\theta^{\prime}
\left[
\begin{array}{c}
1 \\
1
\end{array}
\right]
\theta^3
\left[
\begin{array}{c}
0 \\
1
\end{array}
\right]
}
{
\theta^3
\left[
\begin{array}{c}
\frac12 \\
0
\end{array}
\right]
}.
\end{equation*}
}
\end{theorem}

\begin{proof}
Consider the following elliptic functions:
\begin{equation*}
\varphi(z)
=
\frac
{
\theta^3
\left[
\begin{array}{c}
1 \\
1
\end{array}
\right](z, \tau)
}
{
\theta^2
\left[
\begin{array}{c}
\frac12 \\
1
\end{array}
\right](z, \tau)
\theta
\left[
\begin{array}{c}
0 \\
1
\end{array}
\right](z, \tau)
}
\,\,
\mathrm{and}
\,\,
\psi(z)=
\frac
{
\theta^3
\left[
\begin{array}{c}
1 \\
1
\end{array}
\right](z, \tau)
}
{
\theta^2
\left[
\begin{array}{c}
\frac12 \\
0
\end{array}
\right](z, \tau)
\theta
\left[
\begin{array}{c}
0 \\
1
\end{array}
\right](z, \tau)
}.
\end{equation*}
The theorem can be proved in the same way as Theorem \ref{thm:analogue-Jacobi-(1,1/2), (0,1/2)}.
\end{proof}

\subsection{Derivative formulas for $(\epsilon, \epsilon^{\prime})=(1/2,1/2),  (1/2,3/2)$}

\begin{theorem}
\label{thm:analogue-Jacobi-1/2-1/2,3/2}
{\it
For every $\tau\in\mathbb{H}^2,$ we have
\begin{equation*}
\theta^{\prime}
\left[
\begin{array}{c}
\frac12 \\
\frac12
\end{array}
\right]
=
\frac14
\frac
{
\theta^{\prime}
\left[
\begin{array}{c}
1 \\
1
\end{array}
\right]
\theta^3
\left[
\begin{array}{c}
0 \\
0
\end{array}
\right]
}
{
\theta^3
\left[
\begin{array}{c}
\frac12 \\
\frac12
\end{array}
\right]
}
\,\,
\text{and}
\,\,
\theta^{\prime}
\left[
\begin{array}{c}
\frac12 \\
\frac32
\end{array}
\right]
=
\frac{-1}{4}
\frac
{
\theta^{\prime}
\left[
\begin{array}{c}
1 \\
1
\end{array}
\right]
\theta^3
\left[
\begin{array}{c}
0 \\
0
\end{array}
\right]
}
{
\theta^3
\left[
\begin{array}{c}
\frac12 \\
\frac32
\end{array}
\right]
}.
\end{equation*}
}
\end{theorem}

\begin{proof}
Consider the following elliptic functions:
\begin{equation*}
\varphi(z)
=
\frac
{
\theta^3
\left[
\begin{array}{c}
1 \\
1
\end{array}
\right](z, \tau)
}
{
\theta^2
\left[
\begin{array}{c}
\frac12 \\
\frac12
\end{array}
\right](z, \tau)
\theta
\left[
\begin{array}{c}
0 \\
0
\end{array}
\right](z, \tau)
}
\,\,
\mathrm{and}
\,\,
\psi(z)=
\frac
{
\theta^3
\left[
\begin{array}{c}
1 \\
1
\end{array}
\right](z, \tau)
}
{
\theta^2
\left[
\begin{array}{c}
\frac12 \\
\frac32
\end{array}
\right](z, \tau)
\theta
\left[
\begin{array}{c}
0 \\
0
\end{array}
\right](z, \tau)
}.
\end{equation*}
The theorem can be proved in the same way as Theorem \ref{thm:analogue-Jacobi-(1,1/2), (0,1/2)}.
\end{proof}

\section{Derivative formulas of level 5 }
\label{sec:derivative-level5}

From the discussion of Farkas and Kra \cite[pp. 89-97]{Farkas-Kra},
we have only to consider the case where
\begin{equation*}
(\epsilon, \epsilon^{\prime})=(1/5,j/5), (3/5, j/5), (1,1/5), (1,3/5), (j=1,3,5,7,9).
\end{equation*}

\subsection{Derivative formulas for $(\epsilon,\epsilon^{\prime})=(1/5,1/5), (3/5,3/5)$}

\begin{theorem}
\label{thm:analogue-Jacobi-(1/5,1/5), (3/5,3/5)}
{\it
For every $\tau\in\mathbb{H}^2,$ we have
\begin{equation*}
\frac{
\theta^{\prime}
\left[
\begin{array}{c}
\frac15 \\
\frac15
\end{array}
\right]
}
{
\theta
\left[
\begin{array}{c}
\frac15 \\
\frac15
\end{array}
\right]
}
=
\theta^{\prime}
\left[
\begin{array}{c}
1 \\
1
\end{array}
\right]
\frac
{
\left(
\theta^5
\left[
\begin{array}{c}
\frac15 \\
\frac15
\end{array}
\right]
-
3
\zeta_5^4
\theta^5
\left[
\begin{array}{c}
\frac35 \\
\frac35
\end{array}
\right]
\right)
}
{
10
\theta^3
\left[
\begin{array}{c}
\frac15 \\
\frac15
\end{array}
\right]
\theta^3
\left[
\begin{array}{c}
\frac35 \\
\frac35
\end{array}
\right]
}, \,\,
\frac
{
\theta^{\prime}
\left[
\begin{array}{c}
\frac35 \\
\frac35
\end{array}
\right]
}
{
\theta
\left[
\begin{array}{c}
\frac35 \\
\frac35
\end{array}
\right]
}
=
\theta^{\prime}
\left[
\begin{array}{c}
1 \\
1
\end{array}
\right]
\frac
{
\left(
3
\theta^5
\left[
\begin{array}{c}
\frac15 \\
\frac15
\end{array}
\right]
+
\zeta_5^4
\theta^5
\left[
\begin{array}{c}
\frac35 \\
\frac35
\end{array}
\right]
\right)
}
{
10
\theta^3
\left[
\begin{array}{c}
\frac15 \\
\frac15
\end{array}
\right]
\theta^3
\left[
\begin{array}{c}
\frac35 \\
\frac35
\end{array}
\right]
}.
\end{equation*}
}
\end{theorem}

\begin{proof}
Consider the following elliptic functions:
\begin{equation*}
\varphi(z)
=
\frac
{
\theta^3
\left[
\begin{array}{c}
1 \\
1
\end{array}
\right](z, \tau)
}
{
\theta^2
\left[
\begin{array}{c}
\frac15 \\
\frac15
\end{array}
\right](z, \tau)
\theta
\left[
\begin{array}{c}
\frac35 \\
\frac35
\end{array}
\right](z, \tau)
}
\,\,
\mathrm{and}
\,\,
\psi(z)=
\frac
{
\theta^3
\left[
\begin{array}{c}
1 \\
1
\end{array}
\right](z, \tau)
}
{
\theta^2
\left[
\begin{array}{c}
\frac35 \\
\frac35
\end{array}
\right](z, \tau)
\theta
\left[
\begin{array}{c}
-\frac15 \\
-\frac15
\end{array}
\right](z, \tau)
}.
\end{equation*}
\par
We first note that
in the fundamental parallelogram,
the poles of $\varphi(z)$ are $z=(2\tau+2)/5$ and $z=(\tau+1)/5.$
Direct calculation yields
\begin{equation*}
\mathrm{Res}\left(\varphi(z), \frac{2\tau+2}{5}\right)
=
\zeta_5^3
\frac
{
\theta^3
\left[
\begin{array}{c}
\frac15 \\
\frac15
\end{array}
\right]
}
{
\theta^{\prime}
\left[
\begin{array}{c}
1 \\
1
\end{array}
\right]^2
\theta
\left[
\begin{array}{c}
\frac35 \\
\frac35
\end{array}
\right]
}
\left\{
-3
\frac
{
\theta^{\prime}
\left[
\begin{array}{c}
\frac15 \\
\frac15
\end{array}
\right]
}
{
\theta
\left[
\begin{array}{c}
\frac15 \\
\frac15
\end{array}
\right]
}
+
\frac
{
\theta^{\prime}
\left[
\begin{array}{c}
\frac35 \\
\frac35
\end{array}
\right]
}
{
\theta
\left[
\begin{array}{c}
\frac35 \\
\frac35
\end{array}
\right]
}
\right\}
\end{equation*}
and
\begin{equation*}
\mathrm{Res}\left(\varphi(z), \frac{\tau+1}{5}\right)
=
-
\zeta_5^2
\frac
{
\theta
\left[
\begin{array}{c}
\frac35 \\
\frac35
\end{array}
\right]
}
{
\theta^{\prime}
\left[
\begin{array}{c}
1 \\
1
\end{array}
\right]
}.
\end{equation*}
Since
$
\mathrm{Res}\left(\varphi(z), (2\tau+2)/5\right)+\mathrm{Res}\left(\varphi(z), (\tau+1)/5 \right)=0,
$
it follows that
\begin{equation*}
3
\frac
{
\theta^{\prime}
\left[
\begin{array}{c}
\frac15 \\
\frac15
\end{array}
\right]
}
{
\theta
\left[
\begin{array}{c}
\frac15 \\
\frac15
\end{array}
\right]
}
-
\frac
{
\theta^{\prime}
\left[
\begin{array}{c}
\frac35 \\
\frac35
\end{array}
\right]
}
{
\theta
\left[
\begin{array}{c}
\frac35 \\
\frac35
\end{array}
\right]
}
=
-\zeta_5^4
\frac
{
\theta^{\prime}
\left[
\begin{array}{c}
1 \\
1
\end{array}
\right]
\theta^2
\left[
\begin{array}{c}
\frac35 \\
\frac35
\end{array}
\right]
}
{
\theta^3
\left[
\begin{array}{c}
\frac15 \\
\frac15
\end{array}
\right]
}.
\end{equation*}
\par
From $\psi(z),$ we have
\begin{equation*}
\frac
{
\theta^{\prime}
\left[
\begin{array}{c}
\frac15 \\
\frac15
\end{array}
\right]
}
{
\theta
\left[
\begin{array}{c}
\frac15 \\
\frac15
\end{array}
\right]
}
+3
\frac
{
\theta^{\prime}
\left[
\begin{array}{c}
\frac35 \\
\frac35
\end{array}
\right]
}
{
\theta
\left[
\begin{array}{c}
\frac35 \\
\frac35
\end{array}
\right]
}
=
\frac
{
\theta^{\prime}
\left[
\begin{array}{c}
1 \\
1
\end{array}
\right]
\theta^2
\left[
\begin{array}{c}
\frac15 \\
\frac15
\end{array}
\right]
}
{
\theta^3
\left[
\begin{array}{c}
\frac35 \\
\frac35
\end{array}
\right]
},
\end{equation*}
which proves the theorem.
\end{proof}

\subsection{Derivative formulas for $(\epsilon,\epsilon^{\prime})=(1/5,3/5), (3/5,9/5)$}

\begin{theorem}
\label{thm:analogue-Jacobi-(1/5,3/5), (3/5,9/5)}
{\it
For every $\tau\in\mathbb{H}^2,$ we have
\begin{equation*}
\frac{
\theta^{\prime}
\left[
\begin{array}{c}
\frac15 \\
\frac35
\end{array}
\right]
}
{
\theta
\left[
\begin{array}{c}
\frac15 \\
\frac35
\end{array}
\right]
}
=
-
\theta^{\prime}
\left[
\begin{array}{c}
1 \\
1
\end{array}
\right]
\frac
{
\left(
\theta^5
\left[
\begin{array}{c}
\frac15 \\
\frac35
\end{array}
\right]
+
3
\zeta_5
\theta^5
\left[
\begin{array}{c}
\frac35 \\
\frac95
\end{array}
\right]
\right)
}
{
10
\theta^3
\left[
\begin{array}{c}
\frac15 \\
\frac35
\end{array}
\right]
\theta^3
\left[
\begin{array}{c}
\frac35 \\
\frac95
\end{array}
\right]
},   \,\,
\frac
{
\theta^{\prime}
\left[
\begin{array}{c}
\frac35 \\
\frac95
\end{array}
\right]
}
{
\theta
\left[
\begin{array}{c}
\frac35 \\
\frac95
\end{array}
\right]
}
=
-
\theta^{\prime}
\left[
\begin{array}{c}
1 \\
1
\end{array}
\right]
\frac
{
\left(
3
\theta^5
\left[
\begin{array}{c}
\frac15 \\
\frac35
\end{array}
\right]
-
\zeta_5
\theta^5
\left[
\begin{array}{c}
\frac35 \\
\frac95
\end{array}
\right]
\right)
}
{
10
\theta^3
\left[
\begin{array}{c}
\frac15 \\
\frac35
\end{array}
\right]
\theta^3
\left[
\begin{array}{c}
\frac35 \\
\frac95
\end{array}
\right]
}.
\end{equation*}
}
\end{theorem}

\begin{proof}
Consider the following elliptic functions:
\begin{equation*}
\varphi(z)
=
\frac
{
\theta^3
\left[
\begin{array}{c}
1 \\
1
\end{array}
\right](z, \tau)
}
{
\theta^2
\left[
\begin{array}{c}
\frac15 \\
\frac35
\end{array}
\right](z, \tau)
\theta
\left[
\begin{array}{c}
\frac35 \\
\frac95
\end{array}
\right](z, \tau)
}
\,\,
\mathrm{and}
\,\,
\psi(z)=
\frac
{
\theta^3
\left[
\begin{array}{c}
1 \\
1
\end{array}
\right](z, \tau)
}
{
\theta^2
\left[
\begin{array}{c}
\frac35 \\
\frac95
\end{array}
\right](z, \tau)
\theta
\left[
\begin{array}{c}
-\frac15 \\
-\frac35
\end{array}
\right](z, \tau)
}.
\end{equation*}
The theorem can be proved in the same way as Theorem \ref{thm:analogue-Jacobi-(1/5,1/5), (3/5,3/5)}.
\end{proof}

\subsection{Derivative formulas for $(\epsilon,\epsilon^{\prime})=(1/5,1), (3/5,1)$}

\begin{theorem}
\label{thm:analogue-Jacobi-(1/5,1), (3/5,1)}
{\it
For every $\tau\in\mathbb{H}^2,$ we have
\begin{equation*}
\frac{
\theta^{\prime}
\left[
\begin{array}{c}
\frac15 \\
1
\end{array}
\right]
}
{
\theta
\left[
\begin{array}{c}
\frac15 \\
1
\end{array}
\right]
}
=
-
\zeta_5^3
\theta^{\prime}
\left[
\begin{array}{c}
1 \\
1
\end{array}
\right]
\frac
{
\left(
\theta^5
\left[
\begin{array}{c}
\frac15 \\
1
\end{array}
\right]
+
3
\theta^5
\left[
\begin{array}{c}
\frac35 \\
1
\end{array}
\right]
\right)
}
{
10
\theta^3
\left[
\begin{array}{c}
\frac15 \\
1
\end{array}
\right]
\theta^3
\left[
\begin{array}{c}
\frac35 \\
1
\end{array}
\right]
}, \,\,
\frac
{
\theta^{\prime}
\left[
\begin{array}{c}
\frac35 \\
1
\end{array}
\right]
}
{
\theta
\left[
\begin{array}{c}
\frac35 \\
1
\end{array}
\right]
}
=
-
\zeta_5^3
\theta^{\prime}
\left[
\begin{array}{c}
1 \\
1
\end{array}
\right]
\frac
{
\left(
3
\theta^5
\left[
\begin{array}{c}
\frac15 \\
1
\end{array}
\right]
-
\theta^5
\left[
\begin{array}{c}
\frac35 \\
1
\end{array}
\right]
\right)
}
{
10
\theta^3
\left[
\begin{array}{c}
\frac15 \\
1
\end{array}
\right]
\theta^3
\left[
\begin{array}{c}
\frac35 \\
1
\end{array}
\right]
}.
\end{equation*}
}
\end{theorem}

\begin{proof}
Consider the following elliptic functions:
\begin{equation*}
\varphi(z)
=
\frac
{
\theta^3
\left[
\begin{array}{c}
1 \\
1
\end{array}
\right](z, \tau)
}
{
\theta^2
\left[
\begin{array}{c}
\frac15 \\
1
\end{array}
\right](z, \tau)
\theta
\left[
\begin{array}{c}
\frac35 \\
1
\end{array}
\right](z, \tau)
}
\,\,
\mathrm{and}
\,\,
\psi(z)=
\frac
{
\theta^3
\left[
\begin{array}{c}
1 \\
1
\end{array}
\right](z, \tau)
}
{
\theta^2
\left[
\begin{array}{c}
\frac35 \\
1
\end{array}
\right](z, \tau)
\theta
\left[
\begin{array}{c}
-\frac15 \\
1
\end{array}
\right](z, \tau)
}.
\end{equation*}
From $\varphi(z)$ and $\psi(z),$ we have
\begin{equation}
\label{eqn:relation-(1/5,1)-(3/5,1)-(1)}
3
\frac
{
\theta^{\prime}
\left[
\begin{array}{c}
\frac15 \\
1
\end{array}
\right]
}
{
\theta
\left[
\begin{array}{c}
\frac15 \\
1
\end{array}
\right]
}
-
\frac
{
\theta^{\prime}
\left[
\begin{array}{c}
\frac35 \\
1
\end{array}
\right]
}
{
\theta
\left[
\begin{array}{c}
\frac35 \\
1
\end{array}
\right]
}
=
-\zeta_5^3
\frac
{
\theta^{\prime}
\left[
\begin{array}{c}
1 \\
1
\end{array}
\right]
\theta^2
\left[
\begin{array}{c}
\frac35 \\
1
\end{array}
\right]
}
{
\theta^3
\left[
\begin{array}{c}
\frac15 \\
1
\end{array}
\right]
}
\end{equation}
and
\begin{equation}
\label{eqn:relation-(1/5,1)-(3/5,1)-(2)}
\frac
{
\theta^{\prime}
\left[
\begin{array}{c}
\frac15 \\
1
\end{array}
\right]
}
{
\theta
\left[
\begin{array}{c}
\frac15 \\
1
\end{array}
\right]
}
+3
\frac
{
\theta^{\prime}
\left[
\begin{array}{c}
\frac35 \\
1
\end{array}
\right]
}
{
\theta
\left[
\begin{array}{c}
\frac35 \\
1
\end{array}
\right]
}
=
-\zeta_5^3
\frac
{
\theta^{\prime}
\left[
\begin{array}{c}
1 \\
1
\end{array}
\right]
\theta^2
\left[
\begin{array}{c}
\frac15 \\
1
\end{array}
\right]
}
{
\theta^3
\left[
\begin{array}{c}
\frac35 \\
1
\end{array}
\right]
},
\end{equation}
which proves the theorem.
\end{proof}

By Jacobi's triple product identity (\ref{eqn:Jacobi-triple}),
we obtain the following corollary:

\begin{corollary}
\label{coro:pro-series-(1/5,1)-(3/5,1)}
{\it
For $q\in\mathbb{C}$ with $|q|<1,$
we have
\begin{equation*}
q \prod_{n=1}^{\infty}
\frac
{
(1-q^n)^2
}
{
(1-q^{5n-2})^5 (1-q^{5n-3})^5
}
=
\sum_{n=1}^{\infty}
(d_{1,5}(n)-d_{4,5}(n)) q^n
-
3
\sum_{n=1}^{\infty}
(d_{2,5}(n)-d_{3,5}(n)) q^n
\end{equation*}
and
\begin{equation*}
 \prod_{n=1}^{\infty}
\frac
{
(1-q^n)^2
}
{
(1-q^{5n-1})^5 (1-q^{5n-4})^5
}
=
1
+
3
\sum_{n=1}^{\infty}
(d_{1,5}(n)-d_{4,5}(n)) q^n
+
\sum_{n=1}^{\infty}
(d_{2,5}(n)-d_{3,5}(n)) q^n.
\end{equation*}
}
\end{corollary}

\begin{proof}
The corollary follows from Eqs. (\ref{eqn:relation-(1/5,1)-(3/5,1)-(1)}) and (\ref{eqn:relation-(1/5,1)-(3/5,1)-(2)}).
\end{proof}

\subsection{Derivative formulas for $(\epsilon,\epsilon^{\prime})=(1/5,7/5), (3/5,1/5)$}

\begin{theorem}
\label{thm:analogue-Jacobi-(1/5,7/5), (3/5,1/5)}
{\it
For every $\tau\in\mathbb{H}^2,$ we have
\begin{equation*}
\frac{
\theta^{\prime}
\left[
\begin{array}{c}
\frac15 \\
\frac75
\end{array}
\right]
}
{
\theta
\left[
\begin{array}{c}
\frac15 \\
\frac75
\end{array}
\right]
}
=
-
\zeta_5
\theta^{\prime}
\left[
\begin{array}{c}
1 \\
1
\end{array}
\right]
\frac
{
\left(
\theta^5
\left[
\begin{array}{c}
\frac15 \\
\frac75
\end{array}
\right]
-
3
\zeta_5^3
\theta^5
\left[
\begin{array}{c}
\frac35 \\
\frac15
\end{array}
\right]
\right)
}
{
10
\theta^3
\left[
\begin{array}{c}
\frac15 \\
\frac75
\end{array}
\right]
\theta^3
\left[
\begin{array}{c}
\frac35 \\
\frac15
\end{array}
\right]
},  \,\,
\frac
{
\theta^{\prime}
\left[
\begin{array}{c}
\frac35 \\
\frac15
\end{array}
\right]
}
{
\theta
\left[
\begin{array}{c}
\frac35 \\
\frac15
\end{array}
\right]
}
=
-
\zeta_5
\theta^{\prime}
\left[
\begin{array}{c}
1 \\
1
\end{array}
\right]
\frac
{
\left(
3
\theta^5
\left[
\begin{array}{c}
\frac15 \\
\frac75
\end{array}
\right]
+
\zeta_5^3
\theta^5
\left[
\begin{array}{c}
\frac35 \\
\frac15
\end{array}
\right]
\right)
}
{
10
\theta^3
\left[
\begin{array}{c}
\frac15 \\
\frac75
\end{array}
\right]
\theta^3
\left[
\begin{array}{c}
\frac35 \\
\frac15
\end{array}
\right]
}.
\end{equation*}
}
\end{theorem}

\begin{proof}
Consider the following elliptic functions:
\begin{equation*}
\varphi(z)
=
\frac
{
\theta^3
\left[
\begin{array}{c}
1 \\
1
\end{array}
\right](z, \tau)
}
{
\theta^2
\left[
\begin{array}{c}
\frac15 \\
\frac75
\end{array}
\right](z, \tau)
\theta
\left[
\begin{array}{c}
\frac35 \\
\frac15
\end{array}
\right](z, \tau)
}
\,\,
\mathrm{and}
\,\,
\psi(z)=
\frac
{
\theta^3
\left[
\begin{array}{c}
1 \\
1
\end{array}
\right](z, \tau)
}
{
\theta^2
\left[
\begin{array}{c}
\frac35 \\
\frac15
\end{array}
\right](z, \tau)
\theta
\left[
\begin{array}{c}
-\frac15 \\
\frac35
\end{array}
\right](z, \tau)
}.
\end{equation*}
The theorem can be proved in the same way as Theorem \ref{thm:analogue-Jacobi-(1/5,1/5), (3/5,3/5)}.
\end{proof}

\subsection{Derivative formulas for $(\epsilon,\epsilon^{\prime})=(1/5,9/5), (3/5,7/5)$}

\begin{theorem}
\label{thm:analogue-Jacobi-(1/5,9/5), (3/5,7/5)}
{\it
For every $\tau\in\mathbb{H}^2,$ we have
\begin{equation*}
\frac{
\theta^{\prime}
\left[
\begin{array}{c}
\frac15 \\
\frac95
\end{array}
\right]
}
{
\theta
\left[
\begin{array}{c}
\frac15 \\
\frac95
\end{array}
\right]
}
=
\zeta_5
\theta^{\prime}
\left[
\begin{array}{c}
1 \\
1
\end{array}
\right]
\frac
{
\left(
\theta^5
\left[
\begin{array}{c}
\frac15 \\
\frac95
\end{array}
\right]
-
3
\zeta_5
\theta^5
\left[
\begin{array}{c}
\frac35 \\
\frac75
\end{array}
\right]
\right)
}
{
10
\theta^3
\left[
\begin{array}{c}
\frac15 \\
\frac95
\end{array}
\right]
\theta^3
\left[
\begin{array}{c}
\frac35 \\
\frac75
\end{array}
\right]
}, \,\,
\frac
{
\theta^{\prime}
\left[
\begin{array}{c}
\frac35 \\
\frac75
\end{array}
\right]
}
{
\theta
\left[
\begin{array}{c}
\frac35 \\
\frac75
\end{array}
\right]
}
=
\zeta_5
\theta^{\prime}
\left[
\begin{array}{c}
1 \\
1
\end{array}
\right]
\frac
{
\left(
3
\theta^5
\left[
\begin{array}{c}
\frac15 \\
\frac95
\end{array}
\right]
+
\zeta_5
\theta^5
\left[
\begin{array}{c}
\frac35 \\
\frac75
\end{array}
\right]
\right)
}
{
10
\theta^3
\left[
\begin{array}{c}
\frac15 \\
\frac95
\end{array}
\right]
\theta^3
\left[
\begin{array}{c}
\frac35 \\
\frac75
\end{array}
\right]
}.
\end{equation*}
}
\end{theorem}

\begin{proof}
Consider the following elliptic functions:
\begin{equation*}
\varphi(z)
=
\frac
{
\theta^3
\left[
\begin{array}{c}
1 \\
1
\end{array}
\right](z, \tau)
}
{
\theta^2
\left[
\begin{array}{c}
\frac15 \\
\frac95
\end{array}
\right](z, \tau)
\theta
\left[
\begin{array}{c}
\frac35 \\
-\frac35
\end{array}
\right](z, \tau)
}
\,\,
\mathrm{and}
\,\,
\psi(z)=
\frac
{
\theta^3
\left[
\begin{array}{c}
1 \\
1
\end{array}
\right](z, \tau)
}
{
\theta^2
\left[
\begin{array}{c}
\frac35 \\
\frac75
\end{array}
\right](z, \tau)
\theta
\left[
\begin{array}{c}
-\frac15 \\
\frac15
\end{array}
\right](z, \tau)
}.
\end{equation*}
The theorem can be proved in the same way as Theorem \ref{thm:analogue-Jacobi-(1/5,1/5), (3/5,3/5)}.
\end{proof}

\subsection{Derivative formulas for $(\epsilon,\epsilon^{\prime})=(1,1/5), (1,3/5)$}

\begin{theorem}
\label{thm:analogue-Jacobi-(1,1/5), (1,3/5)}
{\it
For every $\tau\in\mathbb{H}^2,$ we have
\begin{equation*}
\frac{
\theta^{\prime}
\left[
\begin{array}{c}
1 \\
\frac15
\end{array}
\right]
}
{
\theta
\left[
\begin{array}{c}
1 \\
\frac15
\end{array}
\right]
}
=
\theta^{\prime}
\left[
\begin{array}{c}
1 \\
1
\end{array}
\right]
\frac
{
\left(
\theta^5
\left[
\begin{array}{c}
1 \\
\frac15
\end{array}
\right]
-
3
\theta^5
\left[
\begin{array}{c}
1 \\
\frac35
\end{array}
\right]
\right)
}
{
10
\theta^3
\left[
\begin{array}{c}
1 \\
\frac15
\end{array}
\right]
\theta^3
\left[
\begin{array}{c}
1 \\
\frac35
\end{array}
\right]
}, \,\,
\frac
{
\theta^{\prime}
\left[
\begin{array}{c}
1 \\
\frac35
\end{array}
\right]
}
{
\theta
\left[
\begin{array}{c}
1 \\
\frac35
\end{array}
\right]
}
=
\theta^{\prime}
\left[
\begin{array}{c}
1 \\
1
\end{array}
\right]
\frac
{
\left(
3
\theta^5
\left[
\begin{array}{c}
1 \\
\frac15
\end{array}
\right]
+
\theta^5
\left[
\begin{array}{c}
1 \\
\frac35
\end{array}
\right]
\right)
}
{
10
\theta^3
\left[
\begin{array}{c}
1 \\
\frac15
\end{array}
\right]
\theta^3
\left[
\begin{array}{c}
1 \\
\frac35
\end{array}
\right]
}.
\end{equation*}
}
\end{theorem}

\begin{proof}
Consider the following elliptic functions:
\begin{equation*}
\varphi(z)
=
\frac
{
\theta^3
\left[
\begin{array}{c}
1 \\
1
\end{array}
\right](z, \tau)
}
{
\theta^2
\left[
\begin{array}{c}
1 \\
\frac15
\end{array}
\right](z, \tau)
\theta
\left[
\begin{array}{c}
1 \\
\frac35
\end{array}
\right](z, \tau)
}
\,\,
\mathrm{and}
\,\,
\psi(z)=
\frac
{
\theta^3
\left[
\begin{array}{c}
1 \\
1
\end{array}
\right](z, \tau)
}
{
\theta^2
\left[
\begin{array}{c}
1 \\
\frac35
\end{array}
\right](z, \tau)
\theta
\left[
\begin{array}{c}
1 \\
-\frac15
\end{array}
\right](z, \tau)
}.
\end{equation*}
The theorem can be proved in the same way as Theorem \ref{thm:analogue-Jacobi-(1/5,1/5), (3/5,3/5)}.
\end{proof}

\section{Derivative formulas of level 6 }
\label{sec:derivative-level6}

From the discussion of Farkas and Kra \cite[pp. 204-213]{Farkas-Kra},
we have only to consider the case where
\begin{equation*}
(\epsilon, \epsilon^{\prime})=(0,1/3), (0,2/3), (1/3,j/3), (2/3,k/3), (1,2/3),  (j=0,2,4, \,k=0,1,2,3,4,5).
\end{equation*}
In this section,  for $n\in\mathbb{N},$ we set
\begin{equation*}
\left(\frac{n}{3}\right)
=
\begin{cases}
+1, \,\,&\text{if}  \,\,n\equiv 1  \,\,(\mathrm{mod} \, 3),  \\
-1, \,\,&\text{if}  \,\,n\equiv -1 \,\,(\mathrm{mod} \, 3),  \\
0, \,\,&\text{if}  \,\,n\equiv 0 \,\,(\mathrm{mod} \, 3).  \\
\end{cases}
\end{equation*}

\subsection{Derivative formulas for $(\epsilon, \epsilon^{\prime})=(0,1/3), (0,2/3)$}

\begin{theorem}
\label{thm:analogue-Jacobi-(0,1/3), (0,2/3)}
{\it
For every $\tau\in\mathbb{H}^2,$ we have
\begin{equation*}
\theta^{\prime}
\left[
\begin{array}{c}
0 \\
\frac13
\end{array}
\right]
=
\frac13
\frac
{
\theta^{\prime}
\left[
\begin{array}{c}
1 \\
1
\end{array}
\right]
\theta^3
\left[
\begin{array}{c}
1 \\
\frac13
\end{array}
\right]
}
{
\theta^2
\left[
\begin{array}{c}
0 \\
\frac13
\end{array}
\right]
\theta
\left[
\begin{array}{c}
0 \\
1
\end{array}
\right]
}
\,\,
\text{and}
\,\,
\theta^{\prime}
\left[
\begin{array}{c}
0 \\
\frac23
\end{array}
\right]
=
\frac13
\frac
{
\theta^{\prime}
\left[
\begin{array}{c}
1 \\
1
\end{array}
\right]
\theta^3
\left[
\begin{array}{c}
1 \\
\frac13
\end{array}
\right]
}
{
\theta^2
\left[
\begin{array}{c}
0 \\
\frac23
\end{array}
\right]
\theta
\left[
\begin{array}{c}
0 \\
0
\end{array}
\right]
}.
\end{equation*}
}
\end{theorem}

\begin{proof}
Consider the following elliptic functions:
\begin{equation*}
\varphi(z)
=
\frac
{
\theta^3
\left[
\begin{array}{c}
1 \\
1
\end{array}
\right](z, \tau)
}
{
\theta^2
\left[
\begin{array}{c}
0 \\
\frac13
\end{array}
\right](z, \tau)
\theta
\left[
\begin{array}{c}
1 \\
\frac13
\end{array}
\right](z, \tau)
}
\,\,
\mathrm{and}
\,\,
\psi(z)=
\frac
{
\theta^3
\left[
\begin{array}{c}
1 \\
1
\end{array}
\right](z, \tau)
}
{
\theta^2
\left[
\begin{array}{c}
0 \\
\frac23
\end{array}
\right](z, \tau)
\theta
\left[
\begin{array}{c}
1 \\
-\frac13
\end{array}
\right](z, \tau)
}.
\end{equation*}
The theorem can be proved in the same way as Theorem \ref{thm:analogue-Jacobi-(1,1/2), (0,1/2)}.
\end{proof}

By Jacobi's triple product identity (\ref{eqn:Jacobi-triple}),
we obtain the following corollaries:

\begin{corollary}
\label{coro:pro-series-(0,1/3)}
{\it
For every $\tau\in\mathbb{H}^2,$ we have
\begin{equation}
\frac{\eta^5(2\tau)}{\eta^2(\tau)}
=
\sum_{n=1}^{\infty} (-1)^{n-1}n\left(\frac{n}{3}\right) \exp \left(\frac{2\pi i n^2\tau}{3} \right)
\end{equation}
and
\begin{equation}
\frac{\eta(\tau) \eta^6(6\tau)}{ \eta^2(2\tau)\eta^3(3\tau)}
=
\sum_{n=1}^{\infty} (d_{1,3}^{*}(n)-d_{2,3}^{*}(n)) q^n, \quad q=\exp(2\pi i \tau).
\end{equation}
}
\end{corollary}

\begin{proof}
The corollary follows from 
$
\theta^{\prime}
\left[
\begin{array}{c}
0 \\
\frac13
\end{array}
\right]
$
and
$
\theta^{\prime}
\left[
\begin{array}{c}
0 \\
\frac13
\end{array}
\right]/
\theta
\left[
\begin{array}{c}
0 \\
\frac13
\end{array}
\right].
$
\end{proof}

\begin{corollary}
\label{coro:pro-series-(0,2/3)}
{\it
For every $\tau\in\mathbb{H}^2,$ we have
\begin{equation}
\frac{\eta^2(\tau) \eta^2(4\tau)}{\eta(2\tau)}
=
\sum_{n=1}^{\infty} n\left(\frac{n}{3}\right) \exp \left(\frac{2\pi i n^2\tau}{3} \right)
\end{equation}
and
\begin{equation}
\frac{\eta(2\tau) \eta^3(3\tau) \eta^3(12\tau)}{ \eta(\tau)\eta(4\tau) \eta^3(6\tau)}
=
\sum_{n=1}^{\infty} (d_{1,6}^{*}(n)+d_{2,6}^{*}(n)-d_{4,6}^{*}(n)-d_{5,6}^{*}(n)) q^n, \quad q=\exp(2\pi i \tau).
\end{equation}
}
\end{corollary}

\begin{proof}
The corollary follows from 
$
\theta^{\prime}
\left[
\begin{array}{c}
0 \\
\frac23
\end{array}
\right]
$
and
$
\theta^{\prime}
\left[
\begin{array}{c}
0 \\
\frac23
\end{array}
\right]/
\theta
\left[
\begin{array}{c}
0 \\
\frac23
\end{array}
\right].
$
\end{proof}

\subsection{Derivative formulas for $(\epsilon, \epsilon^{\prime})=(1/3,0), (1/3,2/3), (1/3,4/3)$}

\begin{theorem}
\label{thm:analogue-Jacobi-(1/3,0), (1/3,2/3), (1/3,4/3)}
{\it
For every $\tau\in\mathbb{H}^2,$ we have
\begin{equation*}
\theta^{\prime}
\left[
\begin{array}{c}
\frac13 \\
0
\end{array}
\right]
=
-
\frac13
\frac
{
\theta^{\prime}
\left[
\begin{array}{c}
1 \\
1
\end{array}
\right]
\theta^3
\left[
\begin{array}{c}
\frac13 \\
1
\end{array}
\right]
}
{
\theta^2
\left[
\begin{array}{c}
\frac13 \\
0
\end{array}
\right]
\theta
\left[
\begin{array}{c}
1 \\
0
\end{array}
\right]
},
\,\,
\theta^{\prime}
\left[
\begin{array}{c}
\frac13 \\
\frac23
\end{array}
\right]
=
-
\frac13
\frac
{
\theta^{\prime}
\left[
\begin{array}{c}
1 \\
1
\end{array}
\right]
\theta^3
\left[
\begin{array}{c}
\frac13 \\
\frac53
\end{array}
\right]
}
{
\theta^2
\left[
\begin{array}{c}
\frac13 \\
\frac23
\end{array}
\right]
\theta
\left[
\begin{array}{c}
1 \\
0
\end{array}
\right]
},
\,\,
\text{and}
\,\,
\theta^{\prime}
\left[
\begin{array}{c}
\frac13 \\
\frac43
\end{array}
\right]
=
\frac13
\frac
{
\theta^{\prime}
\left[
\begin{array}{c}
1 \\
1
\end{array}
\right]
\theta^3
\left[
\begin{array}{c}
\frac13 \\
\frac13
\end{array}
\right]
}
{
\theta^2
\left[
\begin{array}{c}
\frac13 \\
\frac43
\end{array}
\right]
\theta
\left[
\begin{array}{c}
1 \\
0
\end{array}
\right]
}.
\end{equation*}
}
\end{theorem}

\begin{proof}
Consider the following elliptic functions:
\begin{equation*}
\varphi(z)
=
\frac
{
\theta^3
\left[
\begin{array}{c}
1 \\
1
\end{array}
\right](z, \tau)
}
{
\theta^2
\left[
\begin{array}{c}
\frac13 \\
0
\end{array}
\right](z, \tau)
\theta
\left[
\begin{array}{c}
\frac13 \\
1
\end{array}
\right](z, \tau)
},
\,\,
\phi(z)=
\frac
{
\theta^3
\left[
\begin{array}{c}
1 \\
1
\end{array}
\right](z, \tau)
}
{
\theta^2
\left[
\begin{array}{c}
\frac13 \\
\frac23
\end{array}
\right](z, \tau)
\theta
\left[
\begin{array}{c}
\frac13 \\
\frac53
\end{array}
\right](z, \tau)
},
\end{equation*}
and
\begin{equation*}
\psi(z)=
\frac
{
\theta^3
\left[
\begin{array}{c}
1 \\
1
\end{array}
\right](z, \tau)
}
{
\theta^2
\left[
\begin{array}{c}
\frac13 \\
\frac43
\end{array}
\right](z, \tau)
\theta
\left[
\begin{array}{c}
\frac13 \\
\frac13
\end{array}
\right](z, \tau)
}.
\end{equation*}
The theorem can be proved in the same way as Theorem \ref{thm:analogue-Jacobi-(1,1/2), (0,1/2)}.
\end{proof}

By Jacobi's triple product identity (\ref{eqn:Jacobi-triple}),
we obtain the following corollary:

\begin{corollary}
\label{pro-series-(1/3,0)}
{\it
For every $\tau\in\mathbb{H}^2,$ we have
\begin{equation}
\label{eqn:pro-series-(1/3,0)}
\frac{\eta^5(\tau)}{\eta^2(2\tau)}
=\sum_{n=1, odd}^{\infty} n \left(\frac{n}{3}\right)  \exp \left( \frac{2 \pi i n^2 \tau}{24} \right).
\end{equation}
}
\end{corollary}

\begin{proof}
The corollary follows from the derivative formula for $(\epsilon, \epsilon^{\prime})=(1/3,0).$
\end{proof}

\subsection{Derivative formulas for $(\epsilon, \epsilon^{\prime})=(2/3,0), (2/3,1)$}

\begin{theorem}
\label{thm:analogue-Jacobi-(2/3,0), (2/3,1)}
{\it
For every $\tau\in\mathbb{H}^2,$ we have
\begin{equation*}
\theta^{\prime}
\left[
\begin{array}{c}
\frac23 \\
0
\end{array}
\right]
=
-
\frac13
\frac
{
\theta^{\prime}
\left[
\begin{array}{c}
1 \\
1
\end{array}
\right]
\theta^3
\left[
\begin{array}{c}
\frac13 \\
1
\end{array}
\right]
}
{
\theta^2
\left[
\begin{array}{c}
\frac23 \\
0
\end{array}
\right]
\theta
\left[
\begin{array}{c}
0 \\
0
\end{array}
\right]
}
\,\,
\text{and}
\,\,
\theta^{\prime}
\left[
\begin{array}{c}
\frac23 \\
1
\end{array}
\right]
=
\frac13
\frac
{
\theta^{\prime}
\left[
\begin{array}{c}
1 \\
1
\end{array}
\right]
\theta^3
\left[
\begin{array}{c}
\frac13 \\
1
\end{array}
\right]
}
{
\theta^2
\left[
\begin{array}{c}
\frac23 \\
1
\end{array}
\right]
\theta
\left[
\begin{array}{c}
0 \\
1
\end{array}
\right]
}.
\end{equation*}
}
\end{theorem}

\begin{proof}
Consider the following elliptic functions:
\begin{equation*}
\varphi(z)
=
\frac
{
\theta^3
\left[
\begin{array}{c}
1 \\
1
\end{array}
\right](z, \tau)
}
{
\theta^2
\left[
\begin{array}{c}
\frac23 \\
0
\end{array}
\right](z, \tau)
\theta
\left[
\begin{array}{c}
-\frac13 \\
1
\end{array}
\right](z, \tau)
}
\,\,
\mathrm{and}
\,\,
\psi(z)=
\frac
{
\theta^3
\left[
\begin{array}{c}
1 \\
1
\end{array}
\right](z, \tau)
}
{
\theta^2
\left[
\begin{array}{c}
\frac23 \\
1
\end{array}
\right](z, \tau)
\theta
\left[
\begin{array}{c}
-\frac13 \\
1
\end{array}
\right](z, \tau)
}.
\end{equation*}
The theorem can be proved in the same way as Theorem \ref{thm:analogue-Jacobi-(1,1/2), (0,1/2)}.
\end{proof}

By Jacobi's triple product identity (\ref{eqn:Jacobi-triple}),
we obtain the following corollary:

\begin{corollary}
\label{coro:pro-series-(2/3,1)}
{\it
For every $\tau\in\mathbb{H}^2,$ we have
\begin{equation}
\frac{\eta^5(2\tau)}{\eta^2(\tau)}
=
\sum_{n=1}^{\infty} (-1)^{n-1}n\left(\frac{n}{3}\right) \exp \left(\frac{2\pi i n^2\tau}{3} \right)
\end{equation}
and
\begin{equation}
\frac{\eta^6(2\tau) \eta(3\tau)}{ \eta^3(\tau)\eta^2(6\tau)}
=
1+3
\sum_{n=1}^{\infty} (d_{1,6}(n)-d_{5,6}(n)) q^n, \quad q=\exp(2\pi i \tau).
\end{equation}
}
\end{corollary}

\begin{proof}
The corollary follows from 
$
\theta^{\prime}
\left[
\begin{array}{c}
\frac23 \\
1
\end{array}
\right]
$
and
$
\theta^{\prime}
\left[
\begin{array}{c}
\frac23 \\
1
\end{array}
\right]/
\theta
\left[
\begin{array}{c}
\frac23 \\
1
\end{array}
\right].
$
\end{proof}

\subsection{Derivative formulas for $(\epsilon, \epsilon^{\prime})=(2/3,1/3), (2/3,5/3)$}

\begin{theorem}
\label{thm:analogue-Jacobi-(2/3,1/3), (2/3,5/3)}
{\it
For every $\tau\in\mathbb{H}^2,$ we have
\begin{equation*}
\theta^{\prime}
\left[
\begin{array}{c}
\frac23 \\
\frac13
\end{array}
\right]
=
-
\frac13
\frac
{
\theta^{\prime}
\left[
\begin{array}{c}
1 \\
1
\end{array}
\right]
\theta^3
\left[
\begin{array}{c}
\frac13 \\
\frac53
\end{array}
\right]
}
{
\theta^2
\left[
\begin{array}{c}
\frac23 \\
\frac13
\end{array}
\right]
\theta
\left[
\begin{array}{c}
0 \\
1
\end{array}
\right]
}
\,\,
\text{and}
\,\,
\theta^{\prime}
\left[
\begin{array}{c}
\frac23 \\
\frac53
\end{array}
\right]
=
\frac13
\frac
{
\theta^{\prime}
\left[
\begin{array}{c}
1 \\
1
\end{array}
\right]
\theta^3
\left[
\begin{array}{c}
\frac13 \\
\frac13
\end{array}
\right]
}
{
\theta^2
\left[
\begin{array}{c}
\frac23 \\
\frac53
\end{array}
\right]
\theta
\left[
\begin{array}{c}
0 \\
1
\end{array}
\right]
}.
\end{equation*}
}
\end{theorem}

\begin{proof}
Consider the following elliptic functions:
\begin{equation*}
\varphi(z)
=
\frac
{
\theta^3
\left[
\begin{array}{c}
1 \\
1
\end{array}
\right](z, \tau)
}
{
\theta^2
\left[
\begin{array}{c}
\frac23 \\
\frac13
\end{array}
\right](z, \tau)
\theta
\left[
\begin{array}{c}
-\frac13 \\
\frac13
\end{array}
\right](z, \tau)
}
\,\,
\mathrm{and}
\,\,
\psi(z)=
\frac
{
\theta^3
\left[
\begin{array}{c}
1 \\
1
\end{array}
\right](z, \tau)
}
{
\theta^2
\left[
\begin{array}{c}
\frac23 \\
\frac53
\end{array}
\right](z, \tau)
\theta
\left[
\begin{array}{c}
-\frac13 \\
-\frac13
\end{array}
\right](z, \tau)
}.
\end{equation*}
The theorem can be proved in the same way as Theorem \ref{thm:analogue-Jacobi-(1,1/2), (0,1/2)}.
\end{proof}

\subsection{Derivative formulas for $(\epsilon, \epsilon^{\prime})=(2/3,2/3), (2/3,4/3)$}

\begin{theorem}
\label{thm:analogue-Jacobi-(2/3,2/3), (2/3,4/3)}
{\it
For every $\tau\in\mathbb{H}^2,$ we have
\begin{equation*}
\theta^{\prime}
\left[
\begin{array}{c}
\frac23 \\
\frac23
\end{array}
\right]
=
\frac13
\frac
{
\theta^{\prime}
\left[
\begin{array}{c}
1 \\
1
\end{array}
\right]
\theta^3
\left[
\begin{array}{c}
\frac13 \\
\frac13
\end{array}
\right]
}
{
\theta^2
\left[
\begin{array}{c}
\frac23 \\
\frac23
\end{array}
\right]
\theta
\left[
\begin{array}{c}
0 \\
0
\end{array}
\right]
}
\,\,
\text{and}
\,\,
\theta^{\prime}
\left[
\begin{array}{c}
\frac23 \\
\frac43
\end{array}
\right]
=
\frac13
\frac
{
\theta^{\prime}
\left[
\begin{array}{c}
1 \\
1
\end{array}
\right]
\theta^3
\left[
\begin{array}{c}
\frac13 \\
\frac53
\end{array}
\right]
}
{
\theta^2
\left[
\begin{array}{c}
\frac23 \\
\frac43
\end{array}
\right]
\theta
\left[
\begin{array}{c}
0 \\
0
\end{array}
\right]
}.
\end{equation*}
}
\end{theorem}

\begin{proof}
Consider the following elliptic functions:
\begin{equation*}
\varphi(z)
=
\frac
{
\theta^3
\left[
\begin{array}{c}
1 \\
1
\end{array}
\right](z, \tau)
}
{
\theta^2
\left[
\begin{array}{c}
\frac23 \\
\frac23
\end{array}
\right](z, \tau)
\theta
\left[
\begin{array}{c}
-\frac13 \\
\frac53
\end{array}
\right](z, \tau)
}\,\,
\mathrm{and}
\,\,
\psi(z)=
\frac
{
\theta^3
\left[
\begin{array}{c}
1 \\
1
\end{array}
\right](z, \tau)
}
{
\theta^2
\left[
\begin{array}{c}
\frac23 \\
\frac43
\end{array}
\right](z, \tau)
\theta
\left[
\begin{array}{c}
-\frac13 \\
\frac13
\end{array}
\right](z, \tau)
}.
\end{equation*}
The theorem can be proved in the same way as Theorem \ref{thm:analogue-Jacobi-(1,1/2), (0,1/2)}.
\end{proof}

\subsection{Derivative formulas for $(\epsilon, \epsilon^{\prime})=(1,2/3)$}

\begin{theorem}
\label{thm:analogue-Jacobi-(1,2/3)}
{\it
For every $\tau\in\mathbb{H}^2,$ we have
\begin{equation*}
\theta^{\prime}
\left[
\begin{array}{c}
1 \\
\frac23
\end{array}
\right]
=
\frac13
\frac
{
\theta^{\prime}
\left[
\begin{array}{c}
1 \\
1
\end{array}
\right]
\theta^3
\left[
\begin{array}{c}
1 \\
\frac13
\end{array}
\right]
}
{
\theta^2
\left[
\begin{array}{c}
1 \\
\frac23
\end{array}
\right]
\theta
\left[
\begin{array}{c}
1 \\
0
\end{array}
\right]
}.
\end{equation*}
}
\end{theorem}

\begin{proof}
Consider the following elliptic function:
\begin{equation*}
\varphi(z)
=
\frac
{
\theta^3
\left[
\begin{array}{c}
1 \\
1
\end{array}
\right](z, \tau)
}
{
\theta^2
\left[
\begin{array}{c}
1 \\
\frac23
\end{array}
\right](z, \tau)
\theta
\left[
\begin{array}{c}
1 \\
-\frac13
\end{array}
\right](z, \tau)
}.
\end{equation*}
The theorem can be proved in the same way as Theorem \ref{thm:analogue-Jacobi-(1,1/2), (0,1/2)}.
\end{proof}

\section{Derivative formulas of level 3}
\label{sec:derivative-level3}

From the discussion of Farkas and Kra \cite[pp. 89-97]{Farkas-Kra},
we have only to consider the case where
\begin{equation*}
(\epsilon, \epsilon^{\prime})=(1/3,1/3), (1/3,1), (1/3,5/3), (1,1/3).
\end{equation*}

\subsection{Derivative formula for $(\epsilon, \epsilon^{\prime})=(1/3,1/3)$}

\begin{theorem}
\label{thm:analogue-Jacobi-(1/3,1/3)}
{\it
For every $\tau\in\mathbb{H}^2,$ we have
\begin{align*}
\frac
{
\theta^{\prime}
\left[
\begin{array}{c}
\frac13 \\
\frac13
\end{array}
\right]
}
{
\theta
\left[
\begin{array}{c}
\frac13 \\
\frac13
\end{array}
\right]
}
=&
\frac13
\frac
{
\theta^{\prime}
\left[
\begin{array}{c}
1 \\
1
\end{array}
\right]
\theta^3
\left[
\begin{array}{c}
\frac13 \\
\frac13
\end{array}
\right]
}
{
\theta
\left[
\begin{array}{c}
0 \\
1
\end{array}
\right]
\theta^3
\left[
\begin{array}{c}
\frac23 \\
\frac53
\end{array}
\right]
}
-
\zeta_3
\frac
{
\theta^{\prime}
\left[
\begin{array}{c}
1 \\
1
\end{array}
\right]
\theta
\left[
\begin{array}{c}
0 \\
1
\end{array}
\right]
\theta
\left[
\begin{array}{c}
\frac13 \\
\frac43
\end{array}
\right]
\theta
\left[
\begin{array}{c}
\frac23 \\
\frac23
\end{array}
\right]
}
{
\theta
\left[
\begin{array}{c}
0 \\
0
\end{array}
\right]
\theta
\left[
\begin{array}{c}
1 \\
0
\end{array}
\right]
\theta
\left[
\begin{array}{c}
\frac13 \\
\frac13
\end{array}
\right]
\theta
\left[
\begin{array}{c}
\frac23 \\
\frac53
\end{array}
\right]
}  \\
=&
-
\frac{\pi}{3}
\theta
\left[
\begin{array}{c}
0 \\
0
\end{array}
\right]
\theta
\left[
\begin{array}{c}
1 \\
0
\end{array}
\right]
\frac
{
\theta^3
\left[
\begin{array}{c}
\frac13 \\
\frac13
\end{array}
\right]
}
{
\theta^3
\left[
\begin{array}{c}
\frac23 \\
\frac53
\end{array}
\right]
}
+
\pi
\zeta_3
\theta^2
\left[
\begin{array}{c}
0 \\
1
\end{array}
\right]
\frac
{
\theta
\left[
\begin{array}{c}
\frac13 \\
\frac43
\end{array}
\right]
\theta
\left[
\begin{array}{c}
\frac23 \\
\frac23
\end{array}
\right]
}
{
\theta
\left[
\begin{array}{c}
\frac13 \\
\frac13
\end{array}
\right]
\theta
\left[
\begin{array}{c}
\frac23 \\
\frac53
\end{array}
\right]
}.
\end{align*}
}
\end{theorem}

\begin{proof}
Consider the following elliptic function:
\begin{equation*}
\varphi(z)
=
\frac
{
\theta
\left[
\begin{array}{c}
\frac13 \\
\frac13
\end{array}
\right](z, \tau)
\theta
\left[
\begin{array}{c}
\frac23 \\
\frac53
\end{array}
\right](z, \tau)
\theta
\left[
\begin{array}{c}
0 \\
0
\end{array}
\right](z, \tau)
}
{
\theta^2
\left[
\begin{array}{c}
1 \\
1
\end{array}
\right](z, \tau)
\theta
\left[
\begin{array}{c}
1 \\
0
\end{array}
\right](z, \tau)
}.
\end{equation*}
\par
We first note that
in the fundamental parallelogram,
the poles of $\varphi(z)$ are $z=0$ and $z=1/2.$
Direct calculation yields
\begin{equation*}
\mathrm{Res}\left(\varphi(z), 0 \right)
=
\frac
{
\theta
\left[
\begin{array}{c}
\frac13 \\
\frac13
\end{array}
\right]
\theta
\left[
\begin{array}{c}
\frac23 \\
\frac53
\end{array}
\right]
\theta
\left[
\begin{array}{c}
0 \\
0
\end{array}
\right]
}
{
\theta^{\prime}
\left[
\begin{array}{c}
1 \\
1
\end{array}
\right]^2
\theta
\left[
\begin{array}{c}
1 \\
0
\end{array}
\right]
}
\left\{
\frac
{
\theta^{\prime}
\left[
\begin{array}{c}
\frac13 \\
\frac13
\end{array}
\right]
}
{
\theta
\left[
\begin{array}{c}
\frac13 \\
\frac13
\end{array}
\right]
}
+
\frac
{
\theta^{\prime}
\left[
\begin{array}{c}
\frac23 \\
\frac53
\end{array}
\right]
}
{
\theta
\left[
\begin{array}{c}
\frac23 \\
\frac53
\end{array}
\right]
}
\right\}
\end{equation*}
and
\begin{equation*}
\mathrm{Res}\left(\varphi(z), \frac12 \right)
=
\zeta_3
\frac
{
\theta
\left[
\begin{array}{c}
\frac13 \\
\frac43
\end{array}
\right]
\theta
\left[
\begin{array}{c}
\frac23 \\
\frac23
\end{array}
\right]
\theta
\left[
\begin{array}{c}
0 \\
1
\end{array}
\right]
}
{
\theta^2
\left[
\begin{array}{c}
1 \\
0
\end{array}
\right]
\theta^{\prime}
\left[
\begin{array}{c}
1 \\
1
\end{array}
\right]
}.
\end{equation*}
Since
$
\mathrm{Res}\left(\varphi(z),0 \right)+\mathrm{Res}\left(\varphi(z), \frac12 \right)=0,
$
it follows that
\begin{equation*}
\frac
{
\theta^{\prime}
\left[
\begin{array}{c}
\frac13 \\
\frac13
\end{array}
\right]
}
{
\theta
\left[
\begin{array}{c}
\frac13 \\
\frac13
\end{array}
\right]
}
+
\frac
{
\theta^{\prime}
\left[
\begin{array}{c}
\frac23 \\
\frac53
\end{array}
\right]
}
{
\theta
\left[
\begin{array}{c}
\frac23 \\
\frac53
\end{array}
\right]
}
=
-
\zeta_3
\frac
{
\theta^{\prime}
\left[
\begin{array}{c}
1 \\
1
\end{array}
\right]
\theta
\left[
\begin{array}{c}
0 \\
1
\end{array}
\right]
\theta
\left[
\begin{array}{c}
\frac13 \\
\frac43
\end{array}
\right]
\theta
\left[
\begin{array}{c}
\frac23 \\
\frac23
\end{array}
\right]
}
{
\theta
\left[
\begin{array}{c}
0 \\
0
\end{array}
\right]
\theta
\left[
\begin{array}{c}
1 \\
0
\end{array}
\right]
\theta
\left[
\begin{array}{c}
\frac13 \\
\frac13
\end{array}
\right]
\theta
\left[
\begin{array}{c}
\frac23 \\
\frac53
\end{array}
\right]
},
\end{equation*}
which proves the theorem.
The first equality follows from Theorem \ref{thm:analogue-Jacobi-(2/3,1/3), (2/3,5/3)}.
The second equality follows from Jacobi's derivative formula.
\end{proof}

Considering $\psi(z)=1/\varphi(z),$
we obtain the following theta constant identity:

\begin{theorem}
\label{thm:theta-constant-identity-Jacobi-(1/3,1/3)}
{\it
For every $\tau\in\mathbb{H}^2,$ we have
\begin{equation*}
\theta^2
\left[
\begin{array}{c}
\frac13 \\
\frac13
\end{array}
\right]
\theta^2
\left[
\begin{array}{c}
\frac13 \\
\frac43
\end{array}
\right]
+
\zeta_3^2
\theta^2
\left[
\begin{array}{c}
\frac23 \\
\frac23
\end{array}
\right]
\theta^2
\left[
\begin{array}{c}
\frac23 \\
\frac53
\end{array}
\right]
-
\theta^2
\left[
\begin{array}{c}
0 \\
0
\end{array}
\right]
\theta
\left[
\begin{array}{c}
0 \\
1
\end{array}
\right]
\theta
\left[
\begin{array}{c}
\frac23 \\
\frac53
\end{array}
\right]
=0.
\end{equation*}
}
\end{theorem}

\begin{proof}
Consider the following elliptic function:
\begin{equation*}
\psi(z)
=
\frac
{
\theta^2
\left[
\begin{array}{c}
1 \\
1
\end{array}
\right](z, \tau)
\theta
\left[
\begin{array}{c}
1 \\
0
\end{array}
\right](z, \tau)
}
{
\theta
\left[
\begin{array}{c}
\frac13 \\
\frac13
\end{array}
\right](z, \tau)
\theta
\left[
\begin{array}{c}
\frac23 \\
\frac35
\end{array}
\right](z, \tau)
\theta
\left[
\begin{array}{c}
0 \\
0
\end{array}
\right](z, \tau)
}.
\end{equation*}
\par
We first note that
in the fundamental parallelogram,
the poles of $\psi(z)$ are $z=(\tau+1)/3,$ $z=(\tau-2)/6$ and $(\tau+1)/2.$
Direct calculation yields
\begin{equation*}
\mathrm{Res}\left(\psi(z), \frac{\tau+1}{3} \right)
=
-\zeta_3^2
\frac
{
\theta^2
\left[
\begin{array}{c}
\frac13 \\
\frac13
\end{array}
\right]
\theta
\left[
\begin{array}{c}
\frac13 \\
\frac43
\end{array}
\right]
}
{
\theta^{\prime}
\left[
\begin{array}{c}
1 \\
1
\end{array}
\right]
\theta
\left[
\begin{array}{c}
\frac23 \\
\frac53
\end{array}
\right]
\theta
\left[
\begin{array}{c}
\frac23 \\
\frac23
\end{array}
\right]
},
\end{equation*}
\begin{equation*}
\mathrm{Res}\left(\psi(z), \frac{\tau-2}{6} \right)
=
-\zeta_3
\frac
{
\theta
\left[
\begin{array}{c}
\frac23 \\
\frac53
\end{array}
\right]
\theta
\left[
\begin{array}{c}
\frac23 \\
\frac23
\end{array}
\right]
}
{
\theta^{\prime}
\left[
\begin{array}{c}
1 \\
1
\end{array}
\right]
\theta
\left[
\begin{array}{c}
\frac13 \\
\frac43
\end{array}
\right]
},
\end{equation*}
and
\begin{equation*}
\mathrm{Res}\left(\psi(z), \frac{\tau+1}{2} \right)
=
\zeta_3^2
\frac
{
\theta^2
\left[
\begin{array}{c}
0 \\
0
\end{array}
\right]
\theta
\left[
\begin{array}{c}
0 \\
1
\end{array}
\right]
}
{
\theta^{\prime}
\left[
\begin{array}{c}
1 \\
1
\end{array}
\right]
\theta
\left[
\begin{array}{c}
\frac23 \\
\frac23
\end{array}
\right]
\theta
\left[
\begin{array}{c}
\frac13 \\
\frac43
\end{array}
\right]
}.
\end{equation*}
Since
$
\mathrm{Res}\left(\psi(z),(\tau+1)/3 \right)+\mathrm{Res}\left(\psi(z), (\tau-2)/6 \right)+\mathrm{Res}\left(\psi(z),(\tau+1)/2 \right)=0,
$
the theorem follows.
\end{proof}

\subsection{Derivative formula for $(\epsilon, \epsilon^{\prime})=(1/3,1)$}

\begin{theorem}
\label{thm:analogue-Jacobi-(1/3,1)}
{\it
For every $\tau\in\mathbb{H}^2,$ we have
\begin{align*}
\frac
{
\theta^{\prime}
\left[
\begin{array}{c}
\frac13 \\
1
\end{array}
\right]
}
{
\theta
\left[
\begin{array}{c}
\frac13 \\
1
\end{array}
\right]
}
=&
-
\frac13
\frac
{
\theta^{\prime}
\left[
\begin{array}{c}
1 \\
1
\end{array}
\right]
\theta^3
\left[
\begin{array}{c}
\frac13 \\
1
\end{array}
\right]
}
{
\theta
\left[
\begin{array}{c}
0 \\
1
\end{array}
\right]
\theta^3
\left[
\begin{array}{c}
\frac23 \\
1
\end{array}
\right]
}
+
\frac
{
\theta^{\prime}
\left[
\begin{array}{c}
1 \\
1
\end{array}
\right]
\theta
\left[
\begin{array}{c}
0 \\
1
\end{array}
\right]
\theta
\left[
\begin{array}{c}
\frac13 \\
0
\end{array}
\right]
\theta
\left[
\begin{array}{c}
\frac23 \\
0
\end{array}
\right]
}
{
\theta
\left[
\begin{array}{c}
0 \\
0
\end{array}
\right]
\theta
\left[
\begin{array}{c}
1 \\
0
\end{array}
\right]
\theta
\left[
\begin{array}{c}
\frac13 \\
1
\end{array}
\right]
\theta
\left[
\begin{array}{c}
\frac23 \\
1
\end{array}
\right]
}  \\
=&
\frac{\pi}{3}
\theta
\left[
\begin{array}{c}
0 \\
0
\end{array}
\right]
\theta
\left[
\begin{array}{c}
1 \\
0
\end{array}
\right]
\frac
{
\theta^3
\left[
\begin{array}{c}
\frac13 \\
1
\end{array}
\right]
}
{
\theta^3
\left[
\begin{array}{c}
\frac23 \\
1
\end{array}
\right]
}
-
\pi
\theta^2
\left[
\begin{array}{c}
0 \\
1
\end{array}
\right]
\frac
{
\theta
\left[
\begin{array}{c}
\frac13 \\
0
\end{array}
\right]
\theta
\left[
\begin{array}{c}
\frac23 \\
0
\end{array}
\right]
}
{
\theta
\left[
\begin{array}{c}
\frac13 \\
1
\end{array}
\right]
\theta
\left[
\begin{array}{c}
\frac23 \\
1
\end{array}
\right]
}.
\end{align*}
}
\end{theorem}

\begin{proof}
Consider the following elliptic function:
\begin{equation*}
\varphi(z)
=
\frac
{
\theta
\left[
\begin{array}{c}
\frac13 \\
1
\end{array}
\right](z, \tau)
\theta
\left[
\begin{array}{c}
\frac23 \\
1
\end{array}
\right](z, \tau)
\theta
\left[
\begin{array}{c}
0 \\
0
\end{array}
\right](z, \tau)
}
{
\theta^2
\left[
\begin{array}{c}
1 \\
1
\end{array}
\right](z, \tau)
\theta
\left[
\begin{array}{c}
1 \\
0
\end{array}
\right](z, \tau)
}.
\end{equation*}
\par
We first note that
in the fundamental parallelogram,
the poles of $\varphi(z)$ are $z=0$ and $z=1/2.$
Direct calculation yields
\begin{equation*}
\mathrm{Res}\left(\varphi(z), 0 \right)
=
\frac
{
\theta
\left[
\begin{array}{c}
\frac13 \\
1
\end{array}
\right]
\theta
\left[
\begin{array}{c}
\frac23 \\
1
\end{array}
\right]
\theta
\left[
\begin{array}{c}
0 \\
0
\end{array}
\right]
}
{
\theta^{\prime}
\left[
\begin{array}{c}
1 \\
1
\end{array}
\right]^2
\theta
\left[
\begin{array}{c}
1 \\
0
\end{array}
\right]
}
\left\{
\frac
{
\theta^{\prime}
\left[
\begin{array}{c}
\frac13 \\
1
\end{array}
\right]
}
{
\theta
\left[
\begin{array}{c}
\frac13 \\
1
\end{array}
\right]
}
+
\frac
{
\theta^{\prime}
\left[
\begin{array}{c}
\frac23 \\
1
\end{array}
\right]
}
{
\theta
\left[
\begin{array}{c}
\frac23 \\
1
\end{array}
\right]
}
\right\}
\end{equation*}
and
\begin{equation*}
\mathrm{Res}\left(\varphi(z), \frac12 \right)
=
-
\frac
{
\theta
\left[
\begin{array}{c}
\frac13 \\
0
\end{array}
\right]
\theta
\left[
\begin{array}{c}
\frac23 \\
0
\end{array}
\right]
\theta
\left[
\begin{array}{c}
0 \\
1
\end{array}
\right]
}
{
\theta^2
\left[
\begin{array}{c}
1 \\
0
\end{array}
\right]
\theta^{\prime}
\left[
\begin{array}{c}
1 \\
1
\end{array}
\right]
}.
\end{equation*}
Since
$
\mathrm{Res}\left(\varphi(z),0 \right)+\mathrm{Res}\left(\varphi(z), \frac12 \right)=0,
$
it follows that
\begin{equation*}
\frac
{
\theta^{\prime}
\left[
\begin{array}{c}
\frac13 \\
1
\end{array}
\right]
}
{
\theta
\left[
\begin{array}{c}
\frac13 \\
1
\end{array}
\right]
}
+
\frac
{
\theta^{\prime}
\left[
\begin{array}{c}
\frac23 \\
1
\end{array}
\right]
}
{
\theta
\left[
\begin{array}{c}
\frac23 \\
1
\end{array}
\right]
}
=
\frac
{
\theta^{\prime}
\left[
\begin{array}{c}
1 \\
1
\end{array}
\right]
\theta
\left[
\begin{array}{c}
0 \\
1
\end{array}
\right]
\theta
\left[
\begin{array}{c}
\frac13 \\
0
\end{array}
\right]
\theta
\left[
\begin{array}{c}
\frac23 \\
0
\end{array}
\right]
}
{
\theta
\left[
\begin{array}{c}
0 \\
0
\end{array}
\right]
\theta
\left[
\begin{array}{c}
1 \\
0
\end{array}
\right]
\theta
\left[
\begin{array}{c}
\frac13 \\
1
\end{array}
\right]
\theta
\left[
\begin{array}{c}
\frac23 \\
1
\end{array}
\right]
},
\end{equation*}
which proves the theorem.
The first equality follows from Theorem \ref{thm:analogue-Jacobi-(2/3,0), (2/3,1)}.
The second equality follows from Jacobi's derivative formula.
\end{proof}

Considering $\psi(z)=1/\varphi(z),$
we obtain the following theta constant identity:

\begin{theorem}
\label{thm:theta-constant-identity-Jacobi-(1/3,1)}
{\it
For every $\tau\in\mathbb{H}^2,$ we have
\begin{equation*}
\theta^2
\left[
\begin{array}{c}
\frac13 \\
0
\end{array}
\right]
\theta^2
\left[
\begin{array}{c}
\frac13 \\
1
\end{array}
\right]
+
\zeta_3
\theta^2
\left[
\begin{array}{c}
\frac23 \\
0
\end{array}
\right]
\theta^2
\left[
\begin{array}{c}
\frac23 \\
1
\end{array}
\right]
-
\theta^2
\left[
\begin{array}{c}
0 \\
0
\end{array}
\right]
\theta
\left[
\begin{array}{c}
0 \\
1
\end{array}
\right]
\theta
\left[
\begin{array}{c}
\frac23 \\
1
\end{array}
\right]
=0.
\end{equation*}
}
\end{theorem}

\begin{proof}
Consider the following elliptic function:
\begin{equation*}
\psi(z)
=
\frac
{
\theta^2
\left[
\begin{array}{c}
1 \\
1
\end{array}
\right](z, \tau)
\theta
\left[
\begin{array}{c}
1 \\
0
\end{array}
\right](z, \tau)
}
{
\theta
\left[
\begin{array}{c}
\frac13 \\
1
\end{array}
\right](z, \tau)
\theta
\left[
\begin{array}{c}
\frac23 \\
1
\end{array}
\right](z, \tau)
\theta
\left[
\begin{array}{c}
0 \\
0
\end{array}
\right](z, \tau)
}.
\end{equation*}
The theorem can be proved in the same way as Theorem \ref{thm:theta-constant-identity-Jacobi-(1/3,1/3)}
\end{proof}

\subsection{Derivative formula for $(\epsilon, \epsilon^{\prime})=(1/3,5/3)$}

\begin{theorem}
\label{thm:analogue-Jacobi-(1/3,5/3)}
{\it
For every $\tau\in\mathbb{H}^2,$ we have
\begin{align*}
\frac
{
\theta^{\prime}
\left[
\begin{array}{c}
\frac13 \\
\frac53
\end{array}
\right]
}
{
\theta
\left[
\begin{array}{c}
\frac13 \\
\frac53
\end{array}
\right]
}
=&
\frac13
\frac
{
\theta^{\prime}
\left[
\begin{array}{c}
1 \\
1
\end{array}
\right]
\theta^3
\left[
\begin{array}{c}
\frac13 \\
\frac53
\end{array}
\right]
}
{
\theta
\left[
\begin{array}{c}
0 \\
1
\end{array}
\right]
\theta^3
\left[
\begin{array}{c}
\frac23 \\
\frac13
\end{array}
\right]
}
+
\zeta_3^2
\frac
{
\theta^{\prime}
\left[
\begin{array}{c}
1 \\
1
\end{array}
\right]
\theta
\left[
\begin{array}{c}
0 \\
1
\end{array}
\right]
\theta
\left[
\begin{array}{c}
\frac13 \\
\frac23
\end{array}
\right]
\theta
\left[
\begin{array}{c}
\frac23 \\
\frac43
\end{array}
\right]
}
{
\theta
\left[
\begin{array}{c}
0 \\
0
\end{array}
\right]
\theta
\left[
\begin{array}{c}
1 \\
0
\end{array}
\right]
\theta
\left[
\begin{array}{c}
\frac13 \\
\frac53
\end{array}
\right]
\theta
\left[
\begin{array}{c}
\frac23 \\
\frac13
\end{array}
\right]
}  \\
=&
-
\frac{\pi}{3}
\theta
\left[
\begin{array}{c}
0 \\
0
\end{array}
\right]
\theta
\left[
\begin{array}{c}
1 \\
0
\end{array}
\right]
\frac
{
\theta^3
\left[
\begin{array}{c}
\frac13 \\
\frac53
\end{array}
\right]
}
{
\theta^3
\left[
\begin{array}{c}
\frac23 \\
\frac13
\end{array}
\right]
}
-
\pi
\zeta_3^2
\theta^2
\left[
\begin{array}{c}
0 \\
1
\end{array}
\right]
\frac
{
\theta
\left[
\begin{array}{c}
\frac13 \\
\frac23
\end{array}
\right]
\theta
\left[
\begin{array}{c}
\frac23 \\
\frac43
\end{array}
\right]
}
{
\theta
\left[
\begin{array}{c}
\frac13 \\
\frac53
\end{array}
\right]
\theta
\left[
\begin{array}{c}
\frac23 \\
\frac13
\end{array}
\right]
}.
\end{align*}
}
\end{theorem}

\begin{proof}
Consider the following elliptic function:
\begin{equation*}
\varphi(z)
=
\frac
{
\theta
\left[
\begin{array}{c}
\frac13 \\
\frac53
\end{array}
\right](z, \tau)
\theta
\left[
\begin{array}{c}
\frac23 \\
\frac13
\end{array}
\right](z, \tau)
\theta
\left[
\begin{array}{c}
0 \\
0
\end{array}
\right](z, \tau)
}
{
\theta^2
\left[
\begin{array}{c}
1 \\
1
\end{array}
\right](z, \tau)
\theta
\left[
\begin{array}{c}
1 \\
0
\end{array}
\right](z, \tau)
}.
\end{equation*}
\par
We first note that
in the fundamental parallelogram,
the poles of $\varphi(z)$ are $z=0$ and $z=1/2.$
Direct calculation yields
\begin{equation*}
\mathrm{Res}\left(\varphi(z), 0 \right)
=
\frac
{
\theta
\left[
\begin{array}{c}
\frac13 \\
\frac53
\end{array}
\right]
\theta
\left[
\begin{array}{c}
\frac23 \\
\frac13
\end{array}
\right]
\theta
\left[
\begin{array}{c}
0 \\
0
\end{array}
\right]
}
{
\theta^{\prime}
\left[
\begin{array}{c}
1 \\
1
\end{array}
\right]^2
\theta
\left[
\begin{array}{c}
1 \\
0
\end{array}
\right]
}
\left\{
\frac
{
\theta^{\prime}
\left[
\begin{array}{c}
\frac13 \\
\frac53
\end{array}
\right]
}
{
\theta
\left[
\begin{array}{c}
\frac13 \\
\frac53
\end{array}
\right]
}
+
\frac
{
\theta^{\prime}
\left[
\begin{array}{c}
\frac23 \\
\frac13
\end{array}
\right]
}
{
\theta
\left[
\begin{array}{c}
\frac23 \\
\frac13
\end{array}
\right]
}
\right\}
\end{equation*}
and
\begin{equation*}
\mathrm{Res}\left(\varphi(z), \frac12 \right)
=
-
\zeta_3^2
\frac
{
\theta
\left[
\begin{array}{c}
\frac13 \\
\frac23
\end{array}
\right]
\theta
\left[
\begin{array}{c}
\frac23 \\
\frac43
\end{array}
\right]
\theta
\left[
\begin{array}{c}
0 \\
1
\end{array}
\right]
}
{
\theta^2
\left[
\begin{array}{c}
1 \\
0
\end{array}
\right]
\theta^{\prime}
\left[
\begin{array}{c}
1 \\
1
\end{array}
\right]
}.
\end{equation*}
Since
$
\mathrm{Res}\left(\varphi(z),0 \right)+\mathrm{Res}\left(\varphi(z), \frac12 \right)=0,
$
it follows that
\begin{equation*}
\frac
{
\theta^{\prime}
\left[
\begin{array}{c}
\frac13 \\
\frac53
\end{array}
\right]
}
{
\theta
\left[
\begin{array}{c}
\frac13 \\
\frac53
\end{array}
\right]
}
+
\frac
{
\theta^{\prime}
\left[
\begin{array}{c}
\frac23 \\
\frac13
\end{array}
\right]
}
{
\theta
\left[
\begin{array}{c}
\frac23 \\
\frac13
\end{array}
\right]
}
=
\zeta_3^2
\frac
{
\theta^{\prime}
\left[
\begin{array}{c}
1 \\
1
\end{array}
\right]
\theta
\left[
\begin{array}{c}
0 \\
1
\end{array}
\right]
\theta
\left[
\begin{array}{c}
\frac13 \\
\frac23
\end{array}
\right]
\theta
\left[
\begin{array}{c}
\frac23 \\
\frac43
\end{array}
\right]
}
{
\theta
\left[
\begin{array}{c}
0 \\
0
\end{array}
\right]
\theta
\left[
\begin{array}{c}
1 \\
0
\end{array}
\right]
\theta
\left[
\begin{array}{c}
\frac13 \\
\frac53
\end{array}
\right]
\theta
\left[
\begin{array}{c}
\frac23 \\
\frac13
\end{array}
\right]
},
\end{equation*}
which proves the theorem.
The first equality follows from Theorem \ref{thm:analogue-Jacobi-(2/3,1/3), (2/3,5/3)}.
The second equality follows from Jacobi's derivative formula.
\end{proof}

Considering $\psi(z)=1/\varphi(z),$
we obtain the following theta constant identity:

\begin{theorem}
\label{thm:theta-constant-identity-Jacobi-(1/3,5/3)}
{\it
For every $\tau\in\mathbb{H}^2,$ we have
\begin{equation*}
\theta^2
\left[
\begin{array}{c}
\frac13 \\
\frac23
\end{array}
\right]
\theta^2
\left[
\begin{array}{c}
\frac13 \\
\frac53
\end{array}
\right]
+
\zeta_3^2
\theta^2
\left[
\begin{array}{c}
\frac23 \\
\frac13
\end{array}
\right]
\theta^2
\left[
\begin{array}{c}
\frac23 \\
\frac43
\end{array}
\right]
-
\zeta_3
\theta^2
\left[
\begin{array}{c}
0 \\
0
\end{array}
\right]
\theta
\left[
\begin{array}{c}
0 \\
1
\end{array}
\right]
\theta
\left[
\begin{array}{c}
\frac23 \\
\frac13
\end{array}
\right]
=0.
\end{equation*}
}
\end{theorem}

\begin{proof}
Consider the following elliptic function:
\begin{equation*}
\psi(z)
=
\frac
{
\theta^2
\left[
\begin{array}{c}
1 \\
1
\end{array}
\right](z, \tau)
\theta
\left[
\begin{array}{c}
1 \\
0
\end{array}
\right](z, \tau)
}
{
\theta
\left[
\begin{array}{c}
\frac13 \\
\frac53
\end{array}
\right](z, \tau)
\theta
\left[
\begin{array}{c}
\frac23 \\
\frac13
\end{array}
\right](z, \tau)
\theta
\left[
\begin{array}{c}
0 \\
0
\end{array}
\right](z, \tau)
}.
\end{equation*}
The theorem can be proved in the same way as Theorem \ref{thm:theta-constant-identity-Jacobi-(1/3,1/3)}
\end{proof}

\subsection{Derivative formula for $(\epsilon, \epsilon^{\prime})=(1,1/3)$}

\begin{theorem}
\label{thm:analogue-Jacobi-(1,1/3)}
{\it
For every $\tau\in\mathbb{H}^2,$ we have
\begin{align*}
\frac
{
\theta^{\prime}
\left[
\begin{array}{c}
1 \\
\frac13
\end{array}
\right]
}
{
\theta
\left[
\begin{array}{c}
1 \\
\frac13
\end{array}
\right]
}
=&
-
\frac13
\frac
{
\theta^{\prime}
\left[
\begin{array}{c}
1 \\
1
\end{array}
\right]
\theta^3
\left[
\begin{array}{c}
1 \\
\frac13
\end{array}
\right]
}
{
\theta
\left[
\begin{array}{c}
1 \\
0
\end{array}
\right]
\theta^3
\left[
\begin{array}{c}
1 \\
\frac23
\end{array}
\right]
}
+
\frac
{
\theta^{\prime}
\left[
\begin{array}{c}
1 \\
1
\end{array}
\right]
\theta
\left[
\begin{array}{c}
1 \\
0
\end{array}
\right]
\theta
\left[
\begin{array}{c}
0 \\
\frac13
\end{array}
\right]
\theta
\left[
\begin{array}{c}
0 \\
\frac23
\end{array}
\right]
}
{
\theta
\left[
\begin{array}{c}
0 \\
0
\end{array}
\right]
\theta
\left[
\begin{array}{c}
0\\
1
\end{array}
\right]
\theta
\left[
\begin{array}{c}
1 \\
\frac13
\end{array}
\right]
\theta
\left[
\begin{array}{c}
1 \\
\frac23
\end{array}
\right]
}  \\
=&
\frac{\pi}{3}
\theta
\left[
\begin{array}{c}
0 \\
0
\end{array}
\right]
\theta
\left[
\begin{array}{c}
0 \\
1
\end{array}
\right]
\frac
{
\theta^3
\left[
\begin{array}{c}
1 \\
\frac13
\end{array}
\right]
}
{
\theta^3
\left[
\begin{array}{c}
1 \\
\frac23
\end{array}
\right]
}
-
\pi
\theta^2
\left[
\begin{array}{c}
1 \\
0
\end{array}
\right]
\frac
{
\theta
\left[
\begin{array}{c}
0 \\
\frac13
\end{array}
\right]
\theta
\left[
\begin{array}{c}
0 \\
\frac23
\end{array}
\right]
}
{
\theta
\left[
\begin{array}{c}
1 \\
\frac13
\end{array}
\right]
\theta
\left[
\begin{array}{c}
1 \\
\frac23
\end{array}
\right]
}.
\end{align*}
}
\end{theorem}

\begin{proof}
Consider the following elliptic function:
\begin{equation*}
\varphi(z)
=
\frac
{
\theta
\left[
\begin{array}{c}
1 \\
\frac13
\end{array}
\right](z, \tau)
\theta
\left[
\begin{array}{c}
1 \\
\frac23
\end{array}
\right](z, \tau)
\theta
\left[
\begin{array}{c}
0 \\
1
\end{array}
\right](z, \tau)
}
{
\theta^2
\left[
\begin{array}{c}
1 \\
1
\end{array}
\right](z, \tau)
\theta
\left[
\begin{array}{c}
0 \\
0
\end{array}
\right](z, \tau)
}.
\end{equation*}
\par
We first note that
in the fundamental parallelogram,
the poles of $\varphi(z)$ are $z=0$ and $z=(\tau+1)/2.$
Direct calculation yields
\begin{equation*}
\mathrm{Res}\left(\varphi(z), 0 \right)
=
\frac
{
\theta
\left[
\begin{array}{c}
1 \\
\frac13
\end{array}
\right]
\theta
\left[
\begin{array}{c}
1 \\
\frac23
\end{array}
\right]
\theta
\left[
\begin{array}{c}
0 \\
1
\end{array}
\right]
}
{
\theta^{\prime}
\left[
\begin{array}{c}
1 \\
1
\end{array}
\right]^2
\theta
\left[
\begin{array}{c}
0 \\
0
\end{array}
\right]
}
\left\{
\frac
{
\theta^{\prime}
\left[
\begin{array}{c}
1 \\
\frac13
\end{array}
\right]
}
{
\theta
\left[
\begin{array}{c}
1 \\
\frac13
\end{array}
\right]
}
+
\frac
{
\theta^{\prime}
\left[
\begin{array}{c}
1 \\
\frac23
\end{array}
\right]
}
{
\theta
\left[
\begin{array}{c}
1 \\
\frac23
\end{array}
\right]
}
\right\}
\end{equation*}
and
\begin{equation*}
\mathrm{Res}\left(\varphi(z), \frac{\tau+1}{2} \right)
=
-
\frac
{
\theta
\left[
\begin{array}{c}
0 \\
\frac13
\end{array}
\right]
\theta
\left[
\begin{array}{c}
0 \\
\frac23
\end{array}
\right]
\theta
\left[
\begin{array}{c}
1 \\
0
\end{array}
\right]
}
{
\theta^2
\left[
\begin{array}{c}
0 \\
0
\end{array}
\right]
\theta^{\prime}
\left[
\begin{array}{c}
1 \\
1
\end{array}
\right]
}.
\end{equation*}
Since
$
\mathrm{Res}\left(\varphi(z),0 \right)+\mathrm{Res}\left(\varphi(z), \frac{\tau+1}{2} \right)=0,
$
it follows that
\begin{equation*}
\frac
{
\theta^{\prime}
\left[
\begin{array}{c}
1 \\
\frac13
\end{array}
\right]
}
{
\theta
\left[
\begin{array}{c}
1 \\
\frac13
\end{array}
\right]
}
+
\frac
{
\theta^{\prime}
\left[
\begin{array}{c}
1 \\
\frac23
\end{array}
\right]
}
{
\theta
\left[
\begin{array}{c}
1 \\
\frac23
\end{array}
\right]
}
=
\frac
{
\theta^{\prime}
\left[
\begin{array}{c}
1 \\
1
\end{array}
\right]
\theta
\left[
\begin{array}{c}
1 \\
0
\end{array}
\right]
\theta
\left[
\begin{array}{c}
0 \\
\frac13
\end{array}
\right]
\theta
\left[
\begin{array}{c}
0 \\
\frac23
\end{array}
\right]
}
{
\theta
\left[
\begin{array}{c}
0 \\
0
\end{array}
\right]
\theta
\left[
\begin{array}{c}
0 \\
1
\end{array}
\right]
\theta
\left[
\begin{array}{c}
1 \\
\frac13
\end{array}
\right]
\theta
\left[
\begin{array}{c}
1 \\
\frac23
\end{array}
\right]
},
\end{equation*}
which proves the theorem.
The first equality follows from Theorem \ref{thm:analogue-Jacobi-(1,2/3)}.
The second equality follows from Jacobi's derivative formula.
\end{proof}

Considering $\psi(z)=1/\varphi(z),$
we obtain the following theta constant identity:

\begin{theorem}
\label{thm:theta-constant-identity-Jacobi-(1,1/3)}
{\it
For every $\tau\in\mathbb{H}^2,$ we have
\begin{equation*}
\theta^2
\left[
\begin{array}{c}
1 \\
\frac13
\end{array}
\right]
\theta^2
\left[
\begin{array}{c}
0 \\
\frac23
\end{array}
\right]
-
\theta^2
\left[
\begin{array}{c}
0 \\
\frac13
\end{array}
\right]
\theta^2
\left[
\begin{array}{c}
1 \\
\frac23
\end{array}
\right]
-
\theta^2
\left[
\begin{array}{c}
0 \\
1
\end{array}
\right]
\theta
\left[
\begin{array}{c}
1 \\
0
\end{array}
\right]
\theta
\left[
\begin{array}{c}
1 \\
\frac23
\end{array}
\right]
=0.
\end{equation*}
}
\end{theorem}

\begin{proof}
Consider the following elliptic function:
\begin{equation*}
\psi(z)
=
\frac
{
\theta^2
\left[
\begin{array}{c}
1 \\
1
\end{array}
\right](z, \tau)
\theta
\left[
\begin{array}{c}
0 \\
0
\end{array}
\right](z, \tau)
}
{
\theta
\left[
\begin{array}{c}
1 \\
\frac13
\end{array}
\right](z, \tau)
\theta
\left[
\begin{array}{c}
1 \\
\frac23
\end{array}
\right](z, \tau)
\theta
\left[
\begin{array}{c}
0 \\
1
\end{array}
\right](z, \tau)
}.
\end{equation*}
The theorem can be proved in the same way as Theorem \ref{thm:theta-constant-identity-Jacobi-(1/3,1/3)}
\end{proof}

\section{Derivative formulas of level 8}
\label{sec:derivative-level8}

\subsection{Derivative formulas for $(\epsilon,\epsilon^{\prime})=(0,1/4), (0,3/4)$}

\begin{theorem}
\label{thm:analogue-Jacobi-(0,1/4), (0,3/4)}
{\it
For every $\tau\in\mathbb{H}^2,$ we have
\begin{equation*}
\frac{
\theta^{\prime}
\left[
\begin{array}{c}
0 \\
\frac14
\end{array}
\right]
}
{
\theta
\left[
\begin{array}{c}
0 \\
\frac14
\end{array}
\right]
}
=
\frac
{\theta^{\prime}
\left[
\begin{array}{c}
1 \\
1
\end{array}
\right]
\theta^3
\left[
\begin{array}{c}
1 \\
\frac12
\end{array}
\right]
\left(
\theta^2
\left[
\begin{array}{c}
0 \\
\frac14
\end{array}
\right]
+
3
\theta^2
\left[
\begin{array}{c}
0 \\
\frac34
\end{array}
\right]
\right)
}
{
8
\theta^3
\left[
\begin{array}{c}
0 \\
\frac14
\end{array}
\right]
\theta^3
\left[
\begin{array}{c}
0 \\
\frac34
\end{array}
\right]
}
\end{equation*}
and
\begin{equation*}
\frac
{
\theta^{\prime}
\left[
\begin{array}{c}
0 \\
\frac34
\end{array}
\right]
}
{
\theta
\left[
\begin{array}{c}
0 \\
\frac34
\end{array}
\right]
}
=
\frac
{
\theta^{\prime}
\left[
\begin{array}{c}
1 \\
1
\end{array}
\right]
\theta^3
\left[
\begin{array}{c}
1 \\
\frac12
\end{array}
\right]
\left(
3
\theta^2
\left[
\begin{array}{c}
0 \\
\frac14
\end{array}
\right]
+
\theta^2
\left[
\begin{array}{c}
0 \\
\frac34
\end{array}
\right]
\right)
}
{
8
\theta^3
\left[
\begin{array}{c}
0 \\
\frac14
\end{array}
\right]
\theta^3
\left[
\begin{array}{c}
0 \\
\frac34
\end{array}
\right]
}.
\end{equation*}
}
\end{theorem}

\begin{proof}
Consider the following elliptic functions:
\begin{equation*}
\varphi(z)
=
\frac
{
\theta^3
\left[
\begin{array}{c}
1 \\
1
\end{array}
\right](z, \tau)
}
{
\theta^2
\left[
\begin{array}{c}
0 \\
\frac14
\end{array}
\right](z, \tau)
\theta
\left[
\begin{array}{c}
1 \\
\frac12
\end{array}
\right](z, \tau)
}
\,\,
\mathrm{and}
\,\,
\psi(z)=
\frac
{
\theta^3
\left[
\begin{array}{c}
1 \\
1
\end{array}
\right](z, \tau)
}
{
\theta^2
\left[
\begin{array}{c}
0 \\
\frac34
\end{array}
\right](z, \tau)
\theta
\left[
\begin{array}{c}
1 \\
-\frac12
\end{array}
\right](z, \tau)
}.
\end{equation*}
The theorem can be proved in the same way as Theorem \ref{thm:analogue-Jacobi-(1/5,1/5), (3/5,3/5)}.
\end{proof}

\subsection{Derivative formulas for $(\epsilon,\epsilon^{\prime})=(1,1/4), (1,3/4)$}

\begin{theorem}
\label{thm:analogue-Jacobi-(1,1/4), (1,3/4)}
{\it
For every $\tau\in\mathbb{H}^2,$ we have
\begin{equation*}
\frac{
\theta^{\prime}
\left[
\begin{array}{c}
1 \\
\frac14
\end{array}
\right]
}
{
\theta
\left[
\begin{array}{c}
1 \\
\frac14
\end{array}
\right]
}
=
\frac
{\theta^{\prime}
\left[
\begin{array}{c}
1 \\
1
\end{array}
\right]
\theta^3
\left[
\begin{array}{c}
1 \\
\frac12
\end{array}
\right]
\left(
\theta^2
\left[
\begin{array}{c}
1 \\
\frac14
\end{array}
\right]
-
3
\theta^2
\left[
\begin{array}{c}
1 \\
\frac34
\end{array}
\right]
\right)
}
{
8
\theta^3
\left[
\begin{array}{c}
1 \\
\frac14
\end{array}
\right]
\theta^3
\left[
\begin{array}{c}
1 \\
\frac34
\end{array}
\right]
}
\end{equation*}
and
\begin{equation*}
\frac
{
\theta^{\prime}
\left[
\begin{array}{c}
1 \\
\frac34
\end{array}
\right]
}
{
\theta
\left[
\begin{array}{c}
1 \\
\frac34
\end{array}
\right]
}
=
\frac
{
\theta^{\prime}
\left[
\begin{array}{c}
1 \\
1
\end{array}
\right]
\theta^3
\left[
\begin{array}{c}
1 \\
\frac12
\end{array}
\right]
\left(
3
\theta^2
\left[
\begin{array}{c}
1 \\
\frac14
\end{array}
\right]
-
\theta^2
\left[
\begin{array}{c}
1 \\
\frac34
\end{array}
\right]
\right)
}
{
8
\theta^3
\left[
\begin{array}{c}
1 \\
\frac14
\end{array}
\right]
\theta^3
\left[
\begin{array}{c}
1 \\
\frac34
\end{array}
\right]
}.
\end{equation*}
}
\end{theorem}

\begin{proof}
Consider the following elliptic functions:
\begin{equation*}
\varphi(z)
=
\frac
{
\theta^3
\left[
\begin{array}{c}
1 \\
1
\end{array}
\right](z, \tau)
}
{
\theta^2
\left[
\begin{array}{c}
1 \\
\frac14
\end{array}
\right](z, \tau)
\theta
\left[
\begin{array}{c}
1 \\
\frac12
\end{array}
\right](z, \tau)
}
\,\,
\mathrm{and}
\,\,
\psi(z)=
\frac
{
\theta^3
\left[
\begin{array}{c}
1 \\
1
\end{array}
\right](z, \tau)
}
{
\theta^2
\left[
\begin{array}{c}
1 \\
\frac34
\end{array}
\right](z, \tau)
\theta
\left[
\begin{array}{c}
1 \\
-\frac12
\end{array}
\right](z, \tau)
}.
\end{equation*}
The theorem can be proved in the same way as Theorem \ref{thm:analogue-Jacobi-(1/5,1/5), (3/5,3/5)}.
\end{proof}

\subsection{Preliminary results on number theory}

\begin{theorem}
\label{thm:2squares-2triangular}
{\it
For each $n\in\mathbb{N}_0,$ set
\begin{equation*}
S_2(n)=\sharp
\left\{
(x,y)\in\mathbb{Z}^2  \, | \,
x^2+y^2=n
\right\}, \,\,
T_2(n)=\sharp
\left\{
(x,y)\in\mathbb{Z}^2  \, | \,
t_x+t_y=n
\right\}.
\end{equation*}
Then,
\begin{equation*}
S_2(2n)=S_2(n), \,\,
S_2(4n+1)=T_2(n), \,\,
S_2(4n+3)=0.
\end{equation*}
}
\end{theorem}

\begin{proof}
Set $x=\exp(\pi i \tau).$
By Lemma \ref{lem:Farkas-Kra},
we have
\begin{equation*}
\theta^2
\left[
\begin{array}{c}
0 \\
0
\end{array}
\right](0,\tau)
=
\theta^2
\left[
\begin{array}{c}
0 \\
0
\end{array}
\right](0,2\tau)
+
\theta^2
\left[
\begin{array}{c}
1 \\
0
\end{array}
\right](0,2\tau).
\end{equation*}
Therefore, it follows that
\begin{equation*}
\left(
\sum_{n\in\mathbb{Z}} x^{n^2}
\right)^2
=
\left(
\sum_{n\in\mathbb{Z}} x^{2n^2}
\right)^2
+
\left(
\sum_{n\in\mathbb{Z}} x^{2(n+\frac12)^2}
\right)^2
=
\left(
\sum_{n\in\mathbb{Z}} x^{2n^2}
\right)^2
+
\left(
x^{\frac12}
\sum_{n\in\mathbb{Z}} x^{4\frac{n(n+1)}{2}}
\right)^2,
\end{equation*}
which implies that
\begin{equation*}
\sum_{n=0}^{\infty} S_2(n) x^n
=
\sum_{n=0}^{\infty} S_2(n) x^{2n}
+
\sum_{n=0}^{\infty} T_2(n) x^{4n+1}.
\end{equation*}
The theorem can be obtained by
comparing the coefficients of the terms $x^{2n}, \, x^{4n+1}, \, x^{4n+3}.$
\end{proof}

\begin{lemma}
\label{lem:square-triangular}
{\it
For
every $\tau\in\mathbb{H}^2,$
we have
\begin{equation*}
\theta
\left[
\begin{array}{c}
0 \\
0
\end{array}
\right](0,\tau)
=
\theta
\left[
\begin{array}{c}
0 \\
0
\end{array}
\right](0,4\tau)
+
\theta
\left[
\begin{array}{c}
1 \\
0
\end{array}
\right](0,4\tau).
\end{equation*}
}
\end{lemma}

\begin{proof}
Set $x=\exp(\pi i \tau).$
From the definition,
it follows that
\begin{align*}
\theta
\left[
\begin{array}{c}
0 \\
0
\end{array}
\right](0,\tau)
=&
\sum_{n\in\mathbb{Z}} x^{n^2}
=
\sum_{n\in\mathbb{Z}} x^{(2n)^2}+ \sum_{n\in\mathbb{Z}} x^{(2n+1)^2} \\
=&
\sum_{n\in\mathbb{Z}} x^{4n^2}+ \sum_{n\in\mathbb{Z}} x^{4(n+\frac12)^2}
=
\theta
\left[
\begin{array}{c}
0 \\
0
\end{array}
\right](0,4\tau)
+
\theta
\left[
\begin{array}{c}
1 \\
0
\end{array}
\right](0,4\tau).
\end{align*}
\end{proof}

\begin{theorem}
\label{thm:s12-t12}
{\it
For fixed positive integers $a,b$ and each $n\in\mathbb{N}_0,$
set
\begin{align*}
S_{a,b} (n)=&\,
\sharp
\left\{
(x,y)\in\mathbb{Z}^2 \, | \,
ax^2+by^2=n
\right\}, \,\,
T_{a,b} (n)= \,
\sharp
\left\{
(x,y)\in\mathbb{Z}^2 \, | \,
at_x+b t_y=n
\right\}, \\
M_{a\textrm{-}b} (n)=& \,
\sharp
\left\{
(x,y)\in\mathbb{Z}^2 \, | \,
ax^2+b t_y=n
\right\}.
\end{align*}
Then,
\begin{align*}
&
S_{1,2}(8n+1)=M_{1\textrm{-}1} (n), \,
S_{1,2}(8n+3)=T_{1,2} (n), \,
S_{1,2}(8n+5)=S_{1,2}(8n+7)=0,  \\
&
S_{1,2}(4n)=S_{1,2}(n), \,
S_{1,2}(4n+2)=M_{1\textrm{-}4} (n).
\end{align*}
}
\end{theorem}

\begin{proof}
Set $x=\exp(\pi i \tau).$
From the definition, it follows that
\begin{equation*}
\theta
\left[
\begin{array}{c}
0 \\
0
\end{array}
\right](0,\tau)
\theta
\left[
\begin{array}{c}
0 \\
0
\end{array}
\right](0,2\tau)
=
\left(
\sum_{m\in\mathbb{Z}} x^{m^2}
\right)
\left(
\sum_{n\in\mathbb{Z}} x^{2n^2}
\right)
=
\sum_{n=0}^{\infty} S_{1,2}(n) x^n.
\end{equation*}
\par
From Lemma \ref{lem:square-triangular},
it follows that
\begin{align*}
&
\theta
\left[
\begin{array}{c}
0 \\
0
\end{array}
\right](0,\tau)
\theta
\left[
\begin{array}{c}
0 \\
0
\end{array}
\right](0,2\tau)  \\
=&
\left(
\theta
\left[
\begin{array}{c}
0 \\
0
\end{array}
\right](0,4\tau)
+
\theta
\left[
\begin{array}{c}
1 \\
0
\end{array}
\right](0,4\tau)
\right)
\left(
\theta
\left[
\begin{array}{c}
0 \\
0
\end{array}
\right](0,8\tau)
+
\theta
\left[
\begin{array}{c}
1 \\
0
\end{array}
\right](0,8\tau)
\right)  \\
=&
\theta
\left[
\begin{array}{c}
0 \\
0
\end{array}
\right](0,4\tau)
\theta
\left[
\begin{array}{c}
0 \\
0
\end{array}
\right](0,8\tau)
+
\theta
\left[
\begin{array}{c}
0 \\
0
\end{array}
\right](0,4\tau)
\theta
\left[
\begin{array}{c}
1 \\
0
\end{array}
\right](0,8\tau)  \\
&+
\theta
\left[
\begin{array}{c}
0 \\
0
\end{array}
\right](0,8\tau)
\theta
\left[
\begin{array}{c}
1 \\
0
\end{array}
\right](0,4\tau)
+
\theta
\left[
\begin{array}{c}
1 \\
0
\end{array}
\right](0,4\tau)
\theta
\left[
\begin{array}{c}
1 \\
0
\end{array}
\right](0,8\tau) \\
=&
\left(
\sum_{m\in\mathbb{Z}} x^{4m^2}
\right)
\left(
\sum_{n\in\mathbb{Z}} x^{8n^2}
\right)
+
\left(
\sum_{m\in\mathbb{Z}} x^{4m^2}
\right)
\left(
x^2
\sum_{n\in\mathbb{Z}} x^{16\cdot\frac{n(n+1)}{2}}
\right)  \\
&+
\left(
\sum_{m\in\mathbb{Z}} x^{8m^2}
\right)
\left(
x
\sum_{n\in\mathbb{Z}} x^{8 \cdot \frac{n(n+1)}{2}}
\right)
+
\left(
x
\sum_{m\in\mathbb{Z}} x^{8 \cdot \frac{m(m+1)}{2}}
\right)
\left(
x^2
\sum_{n\in\mathbb{Z}} x^{16\cdot\frac{n(n+1)}{2}}
\right) \\
=&
\sum_{n=0}^{\infty} S_{1,2}(n) x^{4n}
+
\sum_{n=0}^{\infty} M_{1\textrm{-}4}(n) x^{4n+2}
+
\sum_{n=0}^{\infty} M_{1\textrm{-}1}(n) x^{8n+1}
+
\sum_{n=0}^{\infty} T_{1,2}(n) x^{8n+3}.
\end{align*}
\par
Thus, we have
\begin{equation*}
\sum_{n=0}^{\infty} S_{1,2}(n) x^n
=
\sum_{n=0}^{\infty} S_{1,2}(n) x^{4n}
+
\sum_{n=0}^{\infty} M_{1\textrm{-}4}(n) x^{4n+2}
+
\sum_{n=0}^{\infty} M_{1\textrm{-}1}(n) x^{8n+1}
+
\sum_{n=0}^{\infty} T_{1,2}(n) x^{8n+3}.
\end{equation*}
The theorem can be obtained by comparing the coefficients.
\end{proof}

Using Eqs. (\ref{eqn:2-squares}) and (\ref{eqn:1,2-squares}),
we show the following propositions:

\begin{proposition}
\label{prop:0,1/4-pm-3/4}
{\it
For every $\tau\in\mathbb{H}^2,$ we have
\begin{equation}
\label{eqn:log-diff-0-1/4-3/4(1)}
\frac
{
\theta^{\prime}
\left[
\begin{array}{c}
0 \\
\frac14
\end{array}
\right](0,\tau)
}
{
\theta
\left[
\begin{array}{c}
0 \\
\frac14
\end{array}
\right](0,\tau)
}
-
\frac
{
\theta^{\prime}
\left[
\begin{array}{c}
0 \\
\frac34
\end{array}
\right](0,\tau)
}
{
\theta
\left[
\begin{array}{c}
0 \\
\frac34
\end{array}
\right](0,\tau)
}
=
2\pi
\theta^2
\left[
\begin{array}{c}
1 \\
0
\end{array}
\right](0,4\tau)
\end{equation}
and
\begin{equation}
\label{eqn:log-diff-0-1/4-3/4(2)}
\frac
{
\theta^{\prime}
\left[
\begin{array}{c}
0 \\
\frac14
\end{array}
\right](0,\tau)
}
{
\theta
\left[
\begin{array}{c}
0 \\
\frac14
\end{array}
\right](0,\tau)
}
+
\frac
{
\theta^{\prime}
\left[
\begin{array}{c}
0 \\
\frac34
\end{array}
\right](0,\tau)
}
{
\theta
\left[
\begin{array}{c}
0 \\
\frac34
\end{array}
\right](0,\tau)
}
=
-2\sqrt{2}\pi
\theta
\left[
\begin{array}{c}
1 \\
0
\end{array}
\right](0,4\tau)
\theta
\left[
\begin{array}{c}
0 \\
0
\end{array}
\right](0,2\tau).
\end{equation}
}
\end{proposition}

\begin{proof}
Set $x=\exp(\pi i \tau)$ and  $q=x^2=\exp(2\pi i \tau).$
Jacobi's triple product identity (\ref{eqn:Jacobi-triple}) yields
\begin{align*}
\frac{
\theta^{\prime}
\left[
\begin{array}{c}
0 \\
\frac14
\end{array}
\right](0,\tau)
}
{
\theta
\left[
\begin{array}{c}
0 \\
\frac14
\end{array}
\right](0,\tau)
}
=&
-4\pi
\sum_{m,n=1}^{\infty} \sin \left( \frac{3\pi m}{4} \right) x^{n(2m-1)},    \\
\frac{
\theta^{\prime}
\left[
\begin{array}{c}
0 \\
\frac34
\end{array}
\right](0,\tau)
}
{
\theta
\left[
\begin{array}{c}
0 \\
\frac34
\end{array}
\right](0,\tau)
}
=&
-4\pi
\sum_{m,n=1}^{\infty} \sin \left( \frac{\pi m}{4} \right) x^{n(2m-1)}.
\end{align*}
\par
We first treat Eq. (\ref{eqn:log-diff-0-1/4-3/4(1)}).
For this purpose, we have
\begin{align*}
\frac{
\theta^{\prime}
\left[
\begin{array}{c}
0 \\
\frac14
\end{array}
\right](0,\tau)
}
{
\theta
\left[
\begin{array}{c}
0 \\
\frac14
\end{array}
\right](0,\tau)
}
-
\frac{
\theta^{\prime}
\left[
\begin{array}{c}
0 \\
\frac34
\end{array}
\right](0,\tau)
}
{
\theta
\left[
\begin{array}{c}
0 \\
\frac34
\end{array}
\right](0,\tau)
}
=&
-4\pi
\sum_{m,n=1}^{\infty}
\left(
\sin \left( \frac{3\pi m}{4} \right)
-
\sin \left( \frac{\pi m}{4}  \right)
\right)
x^{m(2n-1)}
\\
=&
-8\pi
\sum_{m,n=1}^{\infty}
\cos \left(\frac{\pi m}{2} \right) \sin \left(\frac{\pi m}{4} \right)
x^{m(2n-1)} \\
=&
-8\pi
\sum_{m,n=1}^{\infty}
\cos \left(\pi m \right) \sin \left(\frac{\pi m}{2} \right)
q^{m(2n-1)} \\
=&
8\pi
\sum_{N=0}^{\infty}
\left( d_{1,4}(2N+1)-d_{3,4}(2N+1) \right)
q^{2N+1}   \\
=&
2\pi
\sum_{N=0}^{\infty} S_2(4N+1) q^{4N+1}.
\end{align*}
By Theorem \ref{thm:2squares-2triangular}, we have
\begin{equation*}
\frac{
\theta^{\prime}
\left[
\begin{array}{c}
0 \\
\frac14
\end{array}
\right](0,\tau)
}
{
\theta
\left[
\begin{array}{c}
0 \\
\frac14
\end{array}
\right](0,\tau)
}
-
\frac{
\theta^{\prime}
\left[
\begin{array}{c}
0 \\
\frac34
\end{array}
\right](0,\tau)
}
{
\theta
\left[
\begin{array}{c}
0 \\
\frac34
\end{array}
\right](0,\tau)
}
=
2\pi
\sum_{N=0}^{\infty} T_2(N) q^{4N+1}
=
2\pi
\theta^2
\left[
\begin{array}{c}
1 \\
0
\end{array}
\right](0,4\tau).
\end{equation*}
\par
We then next deal with Eq. (\ref{eqn:log-diff-0-1/4-3/4(2)}).
For this purpose, we have
\begin{align*}
&
\frac{
\theta^{\prime}
\left[
\begin{array}{c}
0 \\
\frac14
\end{array}
\right](0,\tau)
}
{
\theta
\left[
\begin{array}{c}
0 \\
\frac14
\end{array}
\right](0,\tau)
}
+
\frac{
\theta^{\prime}
\left[
\begin{array}{c}
0 \\
\frac34
\end{array}
\right](0,\tau)
}
{
\theta
\left[
\begin{array}{c}
0 \\
\frac34
\end{array}
\right](0,\tau)
}  \\
=&
-4\pi
\sum_{m,n=1}^{\infty}
\left(
\sin \left( \frac{3\pi m}{4} \right)
+
\sin \left( \frac{\pi m}{4}  \right)
\right)
x^{m(2n-1)}
=
-8\pi
\sum_{m,n=1}^{\infty}
\sin \left( \frac{\pi m}{2} \right)
\cos \left( \frac{\pi m}{4}  \right)
x^{m(2n-1)}\\
=&
-4\sqrt{2} \pi
\sum_{N=0}^{\infty}
\left(
d_{1,8}(2N+1)
+
d_{3,8}(2N+1)
-
d_{5,8}(2N+1)
-
d_{7,8}(2N+1)
\right)
x^{2N+1}  \\
=&
-2\sqrt{2} \pi
\sum_{N=0}^{\infty}
S_{1,2}(2N+1)
x^{2N+1}.
\end{align*}
By Theorem \ref{thm:s12-t12},
we obtain
\begin{align*}
&
\frac{
\theta^{\prime}
\left[
\begin{array}{c}
0 \\
\frac14
\end{array}
\right](0,\tau)
}
{
\theta
\left[
\begin{array}{c}
0 \\
\frac14
\end{array}
\right](0,\tau)
}
+
\frac{
\theta^{\prime}
\left[
\begin{array}{c}
0 \\
\frac34
\end{array}
\right](0,\tau)
}
{
\theta
\left[
\begin{array}{c}
0 \\
\frac34
\end{array}
\right](0,\tau)
} \\
=&
-2\sqrt{2} \pi
\left\{
\sum_{N=0}^{\infty}
S_{1,2}(8N+1)
x^{8N+1}
+
\sum_{N=0}^{\infty}
S_{1,2}(8N+3)
x^{8N+3}
\right\}   \\
=&
-2\sqrt{2} \pi
\left\{
\sum_{N=0}^{\infty}
M_{1\textrm{-}1}(N)
x^{8N+1}
+
\sum_{N=0}^{\infty}
T_{1,2}(N)
x^{8N+3}
\right\} \\
=&
-2\sqrt{2} \pi
\left\{
\theta
\left[
\begin{array}{c}
0\\
0
\end{array}
\right](0,8\tau)
\theta
\left[
\begin{array}{c}
1\\
0
\end{array}
\right](0,4\tau)
+
\theta
\left[
\begin{array}{c}
1\\
0
\end{array}
\right](0,4\tau)
\theta
\left[
\begin{array}{c}
1\\
0
\end{array}
\right](0,8\tau)
\right\}   \\
=&
-2\sqrt{2} \pi
\theta
\left[
\begin{array}{c}
1\\
0
\end{array}
\right](0,4\tau)
\left\{
\theta
\left[
\begin{array}{c}
0\\
0
\end{array}
\right](0,8\tau)
+
\theta
\left[
\begin{array}{c}
1\\
0
\end{array}
\right](0,8\tau)
\right\}   \\
=&
-2\sqrt{2} \pi
\theta
\left[
\begin{array}{c}
1\\
0
\end{array}
\right](0,4\tau)
\theta
\left[
\begin{array}{c}
0\\
0
\end{array}
\right](0,2\tau).
\end{align*}
\end{proof}

\subsection{Another derivative formulas for $(\epsilon, \epsilon^{\prime})=(0,1/4), (0,3/4)$ }
\begin{theorem}
\label{thm:analogue-Jacobi-0-1/4,3/4}
{\it
For every $\tau\in\mathbb{H}^2,$
we have
\begin{equation*}
\theta^{\prime}
\left[
\begin{array}{c}
0 \\
\frac14
\end{array}
\right](0,\tau)
=
-
\pi
\theta
\left[
\begin{array}{c}
0 \\
\frac14
\end{array}
\right](0,\tau)
\theta
\left[
\begin{array}{c}
1 \\
0
\end{array}
\right](0,4\tau)
\left\{
\sqrt{2}
\theta
\left[
\begin{array}{c}
0 \\
0
\end{array}
\right](0,2\tau)
-
\theta
\left[
\begin{array}{c}
1 \\
0
\end{array}
\right](0,4\tau)
\right\}
\end{equation*}
and
\begin{equation*}
\theta^{\prime}
\left[
\begin{array}{c}
0 \\
\frac34
\end{array}
\right](0,\tau)
=
-
\pi
\theta
\left[
\begin{array}{c}
0 \\
\frac34
\end{array}
\right](0,\tau)
\theta
\left[
\begin{array}{c}
1 \\
0
\end{array}
\right](0,4\tau)
\left\{
\sqrt{2}
\theta
\left[
\begin{array}{c}
0 \\
0
\end{array}
\right](0,2\tau)
+
\theta
\left[
\begin{array}{c}
1 \\
0
\end{array}
\right](0,4\tau)
\right\}.
\end{equation*}
}
\end{theorem}

\begin{proof}
The theorem follows from Proposition \ref{prop:0,1/4-pm-3/4}.
\end{proof}

\end{document}